\documentclass[12pt]{article} 
\usepackage{amsthm,amssymb,amsmath}
\usepackage{enumerate}
\usepackage{graphicx}
\usepackage{float}
\usepackage{multicol}

\usepackage{mathpazo}
\usepackage{subfig}
\usepackage{tikz}
\usepackage{color}
\definecolor{darkblue}{rgb}{0,0,.4}
\usepackage[bookmarks,
colorlinks=true,
linkcolor=darkblue,
citecolor=darkblue,
urlcolor=darkblue]{hyperref}

\theoremstyle{plain}
\newtheorem{theorem}{Theorem}
\newtheorem{corollary}[theorem]{Corollary}
\newtheorem{lemma}[theorem]{Lemma}
\newtheorem{proposition}[theorem]{Proposition}
\newtheorem{remark}[theorem]{Remark}

\theoremstyle{definition}

\newcommand{\D}{\mathcal{D}}
\newcommand{\nc}{N\!C}

\newcommand{\nn}{N\!N}

\newcommand{\fp}{{\rm fp}}

\newcommand{\exc}{{\rm exc}}
\newcommand{\inv}{{\rm inv}}
\newcommand{\maj}{{\rm maj}}
\newcommand{\crs}{{\rm cr}}
\newcommand{\nes}{{\rm nes}}

\newcommand{\st}{{\rm st}}

\newcommand{\rt}{{\rm rt}}
\newcommand{\lt}{{\rm lt}}
\newcommand{\ct}{{\rm ct}}
\newcommand{\RSK}{R\!S\!K}
\usepackage{fullpage}

\numberwithin{equation}{section} 
\numberwithin{theorem}{section} 
\usepackage{vmargin}
\setmarginsrb{2.5cm}{1.5cm}{1.5cm}{1.5cm}{1.5cm}{0.5cm}{1cm}{0.5cm}
\usepackage{fancyhdr}
\title{Restricted permutations refined by number of crossings and nestings}

\author{Paul M. Rakotomamonjy\\
	\small Department of Mathematics and Computer Science\\[-0.8ex]
	\small Sciences and Technology, BP 906 Antananarivo 101, Madagascar\\[-0.8ex]
	\small \texttt{rpaulmazoto@gmail.com}}
\date{}

\begin{document} 
	\maketitle
	\begin{abstract} 
		Let $\st=\{\st_1,\ldots,\st_k\}$ be a set of $k$ statistics on permutations with  $k\geq 1$. We say that two given subset of permutations $T$ and $T'$ are $\st$-Wilf-equivalent if the joint distributions of all statistics in $\st$ over the sets of $T$-avoiding permutations $S_n(T)$ and $T'$-avoiding permutations $S_n(T')$ are the same. The main purpose of this paper is the (\crs,\nes)-Wilf-equivalence classes for all single patterns in $S_3$, where $\crs$ and $\nes$ denote respectively the statistics number of crossings and nestings. One of the main tools that we use is the bijection $\Theta:S_n(321)\rightarrow S_n(132)$ which was originally exhibited by Elizalde and Pak in \cite{ElizP}. They proved that the bijection $\Theta$ preserves the number of fixed points and excedances.  Since the given formulation of $\Theta$ is not direct, we show that it can be defined directly by a recursive formula. Then, we prove that it also preserves the number of crossings. Due to the fact that the sets of non-nesting permutations and 321-avoiding permutations are the same, these properties of the bijection $\Theta$ leads to an unexpected result related to the q,p-Catalan numbers of Randrianarivony defined in \cite{ARandr}. 		
		\begin{center}
			
		\end{center}
	\end{abstract}
	\section{Introduction and main result}
	Let $E=\{e_1,e_2,\ldots,e_n\}$ be a set of $n$ integers such that $e_1<e_2<\ldots<e_n$. A permutation $\sigma$ of $E$ is a bijection from $E$ to itself which can be written linearly as $\sigma=\sigma(e_1)\sigma(e_2)\ldots \sigma(e_n)$. We shall refer to $|\sigma|:=n$ as the length of $\sigma$. For any permutation $\sigma$ of $E$, the \textit{reduction} of $\sigma$ is $red(\sigma):=\tau\sigma\tau^{-1}$, where $\tau$ is the unique order preserving bijection from $E$ to $[n]:=\{1,2,\ldots,n\}$. Example: $\sigma=43795$ is a permutation of $E=\{3,4,5,7,9\}$ (i.e.~$\sigma(3)=4$, $\sigma(4)=3$, $\sigma(5)=7$,  $\sigma(7)=9$ and $\sigma(9)=5$)  and we have $red(\sigma)=21453$. We will denote by $S_n$ the set of all permutations of $[n]$. 
	
	The concept of crossing and nesting on permutations were introduced by A. de M\'edicis and X.G. Viennot \cite{MedVienot} and several authors extended their study, eg \cite{BMP,Cort,ARandr1, ARandr}. Recently, S. Burrill et al.~\cite{BMP} generalized the definition and introduced the notion of $k$-crossing and $k$-nesting. In this work, we are only interested on 2-crossing and 2-nesting that we simply call crossing and nesting. A \textit{crossing}  in a permutation $\sigma$ is a pair of indexes $(i,j)$ such that $i<j< \sigma(i)<\sigma(j)$  or  $\sigma(i)<\sigma(j)\leq i<j$. A \textit{nesting} of $\sigma$ is similarly defined as a pair $(i,j)$ such that $i<j< \sigma(j)<\sigma(i)$  or  $\sigma(j)<\sigma(i)\leq i<j$. We denote respectively by $\crs(\sigma)$ and $\nes(\sigma)$ the number of crossings and nestings of $\sigma$. For better understanding, one can use arc diagram representations. For $\pi=4\ 6\ 2\ 9\ 8\ 1\ 7\ 10\ 3\ 5\in S_{10}$  (see Fig. \ref{fig:arcdiag}), we have $\crs(\pi)=8$ (upper crossings are $(1,2)$, $(2,4)$, $(2,5)$, $(4,8)$  and lower crossings are $(3,9)$, $(6,9)$, $(6,10)$, $(9,10)$ ) and $\nes(\pi)=4$ (the only upper nesting is $(4,5)$  and the three lower nestings are  $(3,6)$, $(7,9)$ and $(7,10)$).  
	
	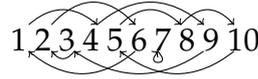
\begin{figure}[h]
		\begin{center}
			\begin{tikzpicture}
			\draw[black] (0,1) node {$1\ 2\ 3\ 4\ 5\ 6\ 7\ 8\ 9\ 10$}; 
			\draw (-1.4,1.2) parabola[parabola height=0.2cm,red] (-0.5,1.2); \draw[->,black] (-0.55,1.25)--(-0.5,1.2);
			\draw (-1.1,1.2) parabola[parabola height=0.3cm,red] (0,1.2); \draw[->,black] (-0.05,1.25)--(0,1.2); 
			\draw (-0.5,1.2) parabola[parabola height=0.3cm,red] (0.9,1.2);\draw[->,black] (0.85,1.25)--(0.9,1.2); 
			\draw (-0.2,1.2) parabola[parabola height=0.2cm,red] (0.6,1.2);\draw[->,black] (0.55,1.25)--(0.6,1.2); 
			\draw (0.6,1.2) parabola[parabola height=0.2cm,red] (1.3,1.2); \draw[->,black] (1.25,1.25)--(1.3,1.2);
			\draw (-0.2,0.85) parabola[parabola height=-0.3cm,red] (1.3,0.85); \draw[<-,black] (-0.2,0.85)--(-0.15,0.8);
			\draw (-0.8,0.85) parabola[parabola height=-0.3cm,red] (0.9,0.85);  \draw[<-,black] (-0.8,0.85)--(-0.75,0.8);
			\draw (-1.4,0.85) parabola[parabola height=-0.3cm,red] (0,0.85);  \draw[<-,black] (-1.4,0.85)--(-1.35,0.8);
			\draw (-1.1,0.85) parabola[parabola height=-0.1cm,red] (-0.8,0.85);  \draw[<-,black] (-1.1,0.85)--(-1.05,0.8);
			\draw (0.2,0.7) -- (0.3,0.85)[rounded corners=0.1cm] -- (0.4,0.7) -- cycle; \draw[->,black] (0.27,0.82)--(0.3,0.85);
			\end{tikzpicture}
			\caption{Arc diagrams of $\pi=4\ 6\ 2\ 9\ 8\ 1\ 7\ 10\ 3\ 5 \in S_{10}$.}
			\label{fig:arcdiag}
		\end{center}
	\end{figure} 
	
	\hspace{-0.58cm}It was known  that the joint distribution of the statistics $\crs$ and $\nes$ over $S_n$ is symmetric (e.g. Corollary 2.3 \cite{ARandr} and Proposition 4 in \cite{Cort}), i.e.~
	$\displaystyle \sum_{\sigma \in S_n}x^{\crs(\sigma)}y^{\nes(\sigma)}=\sum_{\sigma \in S_n}x^{\nes(\sigma)}y^{\crs(\sigma)}$. Combining analytical and bijection methods, the continued fraction expansion of the ordinary generating function of this joint distribution was first computed by Randrianarivony \cite{ARandr}. He obtained the following identity
	\begin{equation*}
	\sum_{n\geq 0}\sum_{\sigma \in S_n}x^{\crs(\sigma)}y^{\nes(\sigma)} z^n=
	\frac{1}{1-\displaystyle\frac{[1]_{x,y}~.z}{				
			1-\displaystyle\frac{[1]_{x,y}~.z}{
				1-\displaystyle\frac{[2]_{x,y}~.z}{								
					1-\displaystyle\frac{[2]_{x,y}~.z}{
						1-\displaystyle\frac{[3]_{x,y}~.z}{
							1-\displaystyle\frac{[3]_{x,y}~.z}{				
								\ddots}
						}
				}}		
	}}}, \label{arthurfc}
	\end{equation*}
	where $[n]_{x,y}=x^{n-1}+x^{n-2}y+\ldots+xy^{n-2}+y^{n-1}$ for any integer $n\geq 1$.
	We will denote respectively by $\nc_{n}$ and  $\nn_{n}$ the set of all noncrossing and nonnesting permutations of $[n]$. It is well known in the literature that $|\nc_{n}|=|\nn_{n}|=C_n$, the $n$-th Catalan number $C_n=\frac{1}{n+1}\binom{2n}{n}$. 
	
	Let us now define some other statistics on permutations. For that, we fix $\sigma\in S_n$. Say that index $i$ is a \textit{fixed point} (resp \textit{excedance}, \textit{descent}) of $\sigma$ if $\sigma(i)  = i$ (resp $\sigma(i)  > i$, $\sigma(i)  > \sigma(i+1)$). An \textit{inversion} of $\sigma$ is a pair of indexes $(i,j)$ such that $i<j$ and $\sigma(i)>\sigma(j)$. We denote  respectively by $\fp(\sigma)$, $\exc(\sigma)$ and $\inv(\sigma)$ the number of fixed points, excedances and inversions of $\sigma$. Define the \textit{major index} of  $\sigma$, denoted by $\maj(\sigma)$, as the sum of all descents of~$\sigma$.
	
	Let $\sigma\in S_n$ and $\tau \in S_k$ with $k\leq n$. Say that the subsequence $\sigma(i_1)\sigma(i_2)\ldots \sigma(i_k)$, with $i_1<i_2<\ldots <i_k$, is an occurrence of $\tau$ if $red[\sigma(i_1)\sigma(i_2)\ldots \sigma(i_k)]=\tau$.  If no such subsequence exists in $\sigma$, we say that it \textit{avoids} $\tau$ or it is $\tau$-\textit{avoiding}. For example, the subsequences $2415$, $2418$, $2416$ and $2417$ are the four occurrences of the pattern $2314$ in $\pi=24135867 \in S_8$. We let the reader to verify that  $\pi$ is $321$-avoiding. For any given subset of permutations $T$, we will denote by $S_n(T)$ the set of all permutations of $[n]$ that avoid all patterns in $T$ and $S(T)=\cup_{n}S_n(T)$. For the case of patterns of length 3, it is well known \cite{Knuth} that regardless of the pattern $\tau\in S_3$, $|S_n(\tau)|=C_n$.  Bijective proof of the fact that $|S_n(321)|=|S_n(132)|$ interested several authors. There exists several constructed bijections between $S_n(321)$ and $S_n(132)$. Each of them has its own properties. See for example \cite{ClaKitaev} and reference therein or \cite{Bloom,ElizP,Knuth}. In this paper, we are particularly interested on the (\fp,\exc)-preserving bijection $\Theta:S_n(321)\rightarrow S_n(132)$ of Elizalde and Pak \cite{ElizP} that will help us for solving our problem.

	Let $\{\st_1, \ldots, \st_k\}$ be a set of $k$ statistics on $S_n$ with $k\geq 1$. Two given subset of patterns $T$ and $T'$ are $(\st_1, \ldots, \st_k)$-\textit{Wilf equivalent} if the following identity holds
	\begin{equation*}
	\sum_{\sigma \in S_n(T)}x_1^{\st_1(\sigma)} \ldots x_k^{\st_k(\sigma)}=\sum_{\sigma \in S_n(T')}x_1^{\st_1(\sigma)} \ldots x_k^{\st_k(\sigma)}.
	\end{equation*}
	Our paper is naturally inspired from some previous known  works. In \cite{ARob}, Robertson et al.~focused on the $\fp$-Wilf-equivalence classes. Elizalde \cite{Eliz1} generalized the result of Robertson et al.~and proposed the  $(\fp,\exc)$-Wilf-equivalence classes. Recently,  Dokos et al.~\cite{Dokos} studied the Wilf-equivalence classes modulo $\inv$ and $\maj$. We summarize in table \ref{table:tab1} their known results for single pattern of length 3.	
	
	\begin{table}[h]
		\begin{center}
			\begin{tabular}{|c|c|c|}
				\hline
				Statistic \st & \st-Wilf-equivalence classes & Reference\\
				\hline
				\fp & $\{132,213,321\}$, $\{231,312\}$ and $\{123\}$ & \cite{ARob}\\
				\hline	
				(\fp,\exc) & $\{132,213,321\}$, $\{231\}$, $\{312\}$ and $\{123\}$ & \cite{Eliz1}\\
				\hline	
				\inv & $\{132,213\}$, $\{231,312\}$, $\{321\}$ and $\{123\}$& \cite{Dokos}\\
				\hline	
				\maj & $\{132,231\}$, $\{213,312\}$, $\{321\}$ and $\{123\}$	&\cite{Dokos}\\
				\hline
				(\fp,\exc,\inv) & $\{132,213\}$, $\{231,312\}$, $\{321\}$ and $\{123\}$& \cite{Eliz1,Dokos}\\
				\hline		
			\end{tabular}
			\caption{Some known results on \fp,\exc,\inv,\maj-Wilf-equivalence classes.}
			\label{table:tab1}
		\end{center}
	\end{table}
	
	In this paper, we focus on the (\crs,\nes)-Wilf-equivalence classes for singleton patterns in $S_3$. Following the notation in \cite{Dokos}, we denote by $[T]_{\st}$ the $\st$-Wilf-equivalent class for any subset of patterns $T$ and set of statistics $\st$. After these necessary preliminaries, we are now in a position to present the main result of this paper that can be stated as follows.
	\begin{theorem} \label{thm:main}  For single patterns in $S_3$, the non singleton $\crs$ and $\nes$-Wilf-equivalence classes are the following
		\begin{itemize}
			\setlength\itemsep{-0.3em}
			\item[{\rm i)}] $[132]_{\nes}=\{132,213\} \text{ and } [231]_{\nes}=\{231,312\}$,
			\item[{\rm ii)}] $[132]_{\crs}= \{132,213,321\}$,
			\item[{\rm iii)}] $[132]_{\crs,\nes}=\{132,213\}$. 
		\end{itemize}
	\end{theorem}
	\hspace{-0.6cm}When we combine this result with those of Elizalde and Dokos et al.~, we get a more generalized one (see  Section \ref{sec4}). The connection with the result of Randrianarivony \cite{ARandr} is due to the fact that nonnesting permutations and 321-avoiding permutations are the same.
	
	The rest of this paper is now organized as follow. In Section \ref{sec2}, we first recall the bijection $\Theta$ of Elizalde and Pak before presenting an interesting inductive formula for it. In Section \ref{sec3}, we establish the $\crs$-preserving of the bijection $\Theta$. In Section \ref{sec4}, we provide the proof of Theorem~\ref{thm:main} using the bijection $\Theta$ and some known trivial bijections on permutations namely reverse, complement and inverse. In Section \ref{sec5}, we discuss the unexpected connection to the $q,p$-Catalan numbers  defined and interpreted by Randrianarivony \cite{ARandr}. In Section~\ref{sec6}, we finally conclude this paper with two interesting remarks. The first one is the correspondence of decomposition between our combinatorial objects while the  second one is the \crs-preserving of the direct bijection $\Gamma$ of Robertson \cite{ARob2} which comes from a recent result of Saracino \cite{Sarino}.
	
	\section{Review of Elizalde and Pak's bijection}\label{sec2}
	As mentioned in introduction, Elizalde and Pak exhibited a bijection $\Theta$ from $S_n(321)$ to $S_n(132)$ which  preserves the statistics $\fp$ and $\exc$. Using Dyck paths as intermediate object, they defined the bijection $\Theta$ as a composition of two bijections. The first one is a bijection  $\Psi$ from $S_n(321)$ to $\mathcal{D}_{n}$ (set of $n$-Dyck paths) and the second one is a bijection  $\Phi$ from $S_n(132)$ to $\mathcal{D}_{n}$. So, we have $\Theta =\Phi^{-1}\circ\Psi$.  In this section,  after some reviews on Dyck paths and notions of tunnels introduced by Elizalde et al.~,  we recall  the two bijections $\Psi$ and $\Phi^{-1}$. Then, we exploit the bijection $\Theta$ and find a new description of its definition.
	
	\subsection{Tunnels on Dyck paths}\label{sec:sec21}\label{sec:sec22}
	A \textit{Dyck path of semi-length $n$}, called also \textit{$n$-Dyck path}, is a path in the first quadrant which starts from the origin $(0,0)$, ends at $(2n, 0)$,  and consists of $n$ \textit{up-steps} $(1,1)$ and $n$ \textit{down-steps} $(1,-1)$.   Usually, we encode each up-step by a letter $u$ and each down-step by $d$. The resulting encoding of a Dyck path is called \textit{Dyck word}. We denote by $\mathcal{D}_{n}$ the set of all $n$-Dyck paths. 	If $D=D_1\ldots D_{2n}\in \D_n$, we have $|D|_u=|D|_d=n$ and $|D(k)|_u\geq |D(k)|_d$ for any initial sub-word $D(k)=D_1\ldots D_{k}$ of length $k$ of $D$, where $|w|_a$ denotes the number of occurrences of letter $a$ in a word $w$. In \cite{ElizD,ElizP},  Elizalde et al.~introduced the statistic number of tunnels on Dyck paths that they used to enumerate restricted permutations according to the number of fixed points and excedances. In fact, a \textit{tunnel} of a Dyck path $D\in \mathcal{D}_{n}$ is an horizontal segment between two lattice points of $D$ that intersects $D$ only at these two points, and stays always below $D$. They distinguished tunnels according to the coordinates of their midpoints.  Graphically, \textit{right} and  \textit{left tunnels} of a Dyck paths are respectively those with their midpoints stay on the right, and on the left of the vertical line through the middle of the path ($x=n$).\textit{ Centered tunnels} are those whose midpoints stay on the vertical line $x=n$.  We denote respectively by $\rt(D)$, $\ct(D)$ and  $\lt(D)$ the number of right tunnels, centered tunnels and  left tunnels. The Dyck path in Fig. \ref{fig:dyckstat} below has the following characteristics: $\rt(D)=3$, $\ct(D)=1$ and $\lt(D)=4$. 
	
	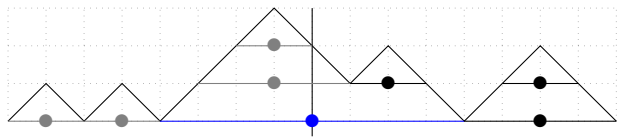
\begin{figure}[h]
		\begin{center}
			\begin{tikzpicture}
			\draw[step=0.5cm, gray, very thin,dotted] (0, 0) grid (8,1.5);
			\draw [-, black](4, -0.2)--(4, 1.5);
			\draw[black] (0, 0)--(0.5, 0.5)--(1, 0)--(1.5, 0.5)--(2, 0)--(2.5, 0.5)--(3, 1)--(3.5, 1.5)--(4, 1)--(4.5, 0.5)--(5, 1)--(5.5, 0.5)--(6, 0)--(6.5, 0.5)--(7, 1)--(7.5, 0.5)--(8, 0);	
			\draw[gray] (0.5,0) node {$\bullet$};  \draw [-,thin, gray](0,0)--(1,0);
			\draw[gray] (1.5,0) node {$\bullet$};  \draw [-,thin, gray](1,0)--(2,0);
			\draw[gray] (3.5,1) node {$\bullet$};  \draw [-,thin, gray](3,1)--(4,1);
			\draw[gray] (3.5,0.5) node {$\bullet$};  \draw [-,thin, gray](2.5,0.5)--(4.5,0.5);
			\draw[blue] (4,0) node {$\bullet$};  \draw [-,thin, blue](2,0)--(6,0);
			\draw[black] (5,0.5) node {$\bullet$};  \draw [-,thin, black](4.5,0.5)--(5.5,0.5);
			\draw[black] (7,0.5) node {$\bullet$};  \draw [-,thin, black](6.5,0.5)--(7.5,0.5);
			\draw[black] (7,0) node {$\bullet$};  \draw [-,thin, black](6,0)--(8,0);
			\end{tikzpicture}
			\caption{Counting tunnels of the Dyck path $D=ududuuuddudduudd$.}\label{fig:dyckstat}
		\end{center}
	\end{figure} 
	
	As seen in figure \ref{fig:dyckstat}, each tunnel is a segment that goes from the beginning of an up-step $u$ to the end of a down-step $d$. As mentioned in \cite{ElizD}, such tunnel is also in obvious one-to-one correspondence with decomposition of the Dyck word $D = AuBdC$, where $B$  and $AC$ are both Dyck paths. 
	
	\subsection{The bijection $\Psi:S_n(321)\rightarrow \mathcal{D}_{n}$}	
	The bijection $\Psi$ is essentially due to Knuth \cite{Knuth} and is a composition of two bijections that we describe here. 
	
	The first is the Robinson-Schensted-Knuth correspondence or simply $\RSK$ correspondence. It is a bijection between $S_n$  and pairs of standard Young tableaux of identical shape $\lambda \vdash n$. This $\RSK$ is based on the insertion algorithm known as $\RSK$ algorithm (see \cite{Knuth,STAN}). Let $\sigma \in S_n$ and $(P,Q)=\RSK(\sigma)$. The tableau $P$ is known as the insertion tableau, and $Q$ the recording tableau. The insertion tableau $P$ is obtained by reading the permutation $\sigma$ from left to right and, at each step, inserting $\sigma(i)$ to the partial tableau obtained so far.  Assume that $(P^{(i-1)},Q^{(i-1)})=\RSK(\sigma(1)\ldots\sigma(i-1))$. We obtain $(P^{(i)},Q^{(i)})$ by inserting 
	$(i, \sigma(i))$ in $(P^{(i-1)},Q^{(i-1)})$, i.e.~inserting $\sigma(i)$ in $P^{(i-1)}$ and $i$ in $Q^{(i-1)}$. Insertion follows the following rules:  if $\sigma(i)$ is larger than all of elements on the first row of $P^{(i-1)}$, then place $\sigma(i)$ at the end of the first row of $P^{(i-1)}$. Otherwise, it takes the place of the leftmost element $x$ on the first row that is larger than $\sigma(i)$ (we say that $x$ is bumped by $\sigma(i)$) and then insert $x$ in the second row by the same way.  The recording tableau $Q^{(i)}$ has the same shape as $P^{(i)}$ and is obtained by placing $i$ in the position of the square that was created at step $i$ on insertion of $\sigma(i)$. By this way, we get $(P,Q)=(P^{(n)},Q^{(n)})$. One of the known properties of $\RSK$ is that the number of rows of $P$ (and as well $Q$) is equals to the length of the longest decreasing subsequence of $\sigma$. Consequently, $\sigma$ is 321-avoiding if and only if P has at most two rows. The duality is also among the famous properties of the $\RSK$ correspondence. It says that
	$\RSK(\sigma^{-1})=(Q,P)$ if and only if $\RSK(\sigma)=(P,Q)$ (e.g. \cite{Knuth}). Below is an example of the construction of $(P,Q)$ from a 321-avoiding permutation $\pi=24135867$ (see Fig. \ref{fig:rsk}).
	
	\begin{figure}[h]
		\begin{center}
			\begin{tikzpicture}
			\draw (0,0.6)-- (0,1);\draw (0.4,0.6)-- (0.4,1); \draw (0,0.6)-- (0.4,0.6);\draw (0,1)-- (0.4,1);
			\draw[step=0.4cm, black, very thin] (0, -0.4) grid (0.4,0);
			\draw[black] (0.2,0.8) node {$2$};
			\draw[black] (0.2,-0.2) node {$1$};
			
			\draw (0.8, 0.6) -- (1.6,0.6); \draw (0.8, 1) -- (1.6,1); \draw (0.8, 0.6) -- (0.8,1); \draw (1.6, 0.6) -- (1.6,1); 
			\draw (1.6, 0.6) -- (1.6,1); \draw (1.2, 0.6) -- (1.2,1);
			\draw (0.8, -0.4) -- (1.6,-0.4);\draw (0.8, 0) -- (1.6,0);\draw (0.8, -0.4) -- (0.8,0); \draw (1.6, -0.4) -- (1.6,0); 
			\draw (1.6, -0.4) -- (1.6,0); \draw (1.2, -0.4) -- (1.2,0);
			\draw[black] (1,0.8) node {$2$};\draw[black] (1.4,0.8) node {$4$};
			\draw[black] (1,-0.2) node {$1$}; \draw[black] (1.4,-0.2) node {$2$};
			
			\draw (2, 0.6) -- (2.8,0.6);\draw (2, 1) -- (2.8,1);\draw (2, 0.2) -- (2,1); \draw (2.4, 0.2) -- (2.4,1); 
			\draw (2.8, 0.6) -- (2.8,1); \draw (2, 0.2) -- (2.4,0.2);
			\draw (2, -0.4) -- (2.8,-0.4);\draw (2, 0) -- (2.8,0);\draw (2, -0.8) -- (2,0); \draw (2.4, -0.8) -- (2.4,0); 
			\draw (2.8, -0.4) -- (2.8,0); \draw (2, -0.8) -- (2.4,-0.8);
			\draw[black] (2.2,0.8) node {$1$};\draw[black] (2.6,0.8) node {$4$};
			\draw[black] (2.2,0.4) node {$2$};
			\draw[black] (2.2,-0.2) node {$1$}; \draw[black] (2.6,-0.2) node {$2$};
			\draw[black] (2.2,-0.6) node {$3$};			
			
			\draw (3.2, 0.6) -- (4,0.6);\draw (3.2, 1) -- (4,1);\draw (3.2, 0.2) -- (3.2,1); \draw (3.6, 0.2) -- (3.6,1); 
			\draw (4, 0.2) -- (4,1); \draw (3.2, 0.2) -- (4,0.2);
			\draw (3.2, -0.4) -- (4,-0.4);\draw (3.2, 0) -- (4,0);\draw (3.2, -0.8) -- (3.2,0); \draw (3.6, -0.8) -- (3.6,0); 
			\draw (4, -0.8) -- (4,0); \draw (3.2, -0.8) -- (4,-0.8);
			\draw[black] (3.4,0.8) node {$1$};  \draw[black] (3.8,0.8) node {$3$};
			\draw[black] (3.4,0.4) node {$2$};  \draw[black] (3.8,0.4) node {$4$};
			\draw[black] (3.4,-0.2) node {$1$}; \draw[black] (3.8,-0.2) node {$2$};
			\draw[black] (3.4,-0.6) node {$3$}; \draw[black] (3.8,-0.6) node {$4$};
			
			\draw (4.6, 0.6) -- (5.8,0.6);\draw (4.6, 1) -- (5.8,1);\draw (4.6, 0.2) -- (5.4,0.2);
			\draw (4.6, 0.2) -- (4.6,1); \draw (5, 0.2) -- (5,1);  \draw (5.4, 0.2) -- (5.4,1);  \draw (5.8, 0.6) -- (5.8,1);
			\draw (4.6, -0.4) -- (5.8,-0.4);\draw (4.6, 0) -- (5.8,0);  \draw (4.6, -0.8) -- (5.4,-0.8);
			\draw (4.6, -0.8) -- (4.6,0); \draw (5, -0.8) -- (5,0); 	  \draw (5.4, -0.8) -- (5.4,0); \draw (5.8, -0.4) -- (5.8,0);
			\draw[black] (4.8,0.8) node {$1$};  \draw[black] (5.2,0.8) node {$3$}; \draw[black] (5.6,0.8) node {$5$};
			\draw[black] (4.8,0.4) node {$2$};  \draw[black] (5.2,0.4) node {$4$};
			\draw[black] (4.8,-0.2) node {$1$}; \draw[black] (5.2,-0.2) node {$2$}; \draw[black] (5.6,-0.2) node {$5$};
			\draw[black] (4.8,-0.6) node {$3$}; \draw[black] (5.2,-0.6) node {$4$};			
			
			\draw (6.2, 0.6) -- (7.8,0.6);\draw (6.2, 1) -- (7.8,1);\draw (6.2, 0.2) -- (7,0.2); 
			\draw (6.2, 0.2) -- (6.2,1); \draw (6.6, 0.2) -- (6.6,1);  \draw (7, 0.2) -- (7,1); \draw (7.4, 0.6) -- (7.4,1); \draw (7.8, 0.6) -- (7.8,1);
			\draw (6.2, -0.4) -- (7.8,-0.4);\draw (6.2, 0) -- (7.8,0);  \draw (6.2, -0.8) -- (7,-0.8);
			\draw (6.2, -0.8) -- (6.2,0); \draw (6.6, -0.8) -- (6.6,0); 	  \draw (7, -0.8) -- (7,0);\draw (7.4, -0.4) -- (7.4,0); \draw (7.8, -0.4) -- (7.8,0);
			\draw[black] (6.4,0.8) node {$1$};  \draw[black] (6.8,0.8) node {$3$}; \draw[black] (7.2,0.8) node {$5$}; \draw[black] (7.6,0.8) node {$8$};
			\draw[black] (6.4,0.4) node {$2$};  \draw[black] (6.8,0.4) node {$4$};
			\draw[black] (6.4,-0.2) node {$1$}; \draw[black] (6.8,-0.2) node {$2$}; \draw[black] (7.2,-0.2) node {$5$}; \draw[black] (7.6,-0.2) node {$6$};
			\draw[black] (6.4,-0.6) node {$3$}; \draw[black] (6.8,-0.6) node {$4$};

			\draw (8.4, 0.6) -- (10,0.6); \draw (8.4, 1) -- (10,1);\draw (8.4, 0.2) -- (9.6,0.2); 
			\draw (8.4, 0.2) -- (8.4,1); \draw (8.8, 0.2) -- (8.8,1);  \draw (9.2, 0.2) -- (9.2,1); \draw (9.6, 0.2) -- (9.6,1); \draw (10, 0.6) -- (10,1);	  	  
			\draw (8.4, -0.4) -- (10,-0.4);\draw (8.4, 0) -- (10,0);  \draw (8.4, -0.8) -- (9.6,-0.8);
			\draw (8.4, -0.8) -- (8.4,0); \draw (8.8, -0.8) -- (8.8,0); 	\draw (9.2, -0.8) -- (9.2,0);  \draw (9.6, -0.8) -- (9.6,0);\draw (9.6, -0.4) -- (9.6,0); \draw (10, -0.4) -- (10,0);
			
			\draw[black] (8.6,0.8) node {$1$};  \draw[black] (9,0.8) node {$3$}; \draw[black] (9.4,0.8) node {$5$}; \draw[black] (9.8,0.8) node {$6$};
			\draw[black] (8.6,0.4) node {$2$};  \draw[black] (9,0.4) node {$4$}; \draw[black] (9.4,0.4) node {$8$};
			\draw[black] (8.6,-0.2) node {$1$}; \draw[black] (9,-0.2) node {$2$}; \draw[black] (9.4,-0.2) node {$5$}; \draw[black] (9.8,-0.2) node {$6$};
			\draw[black] (8.6,-0.6) node {$3$}; \draw[black] (9,-0.6) node {$4$};\draw[black] (9.4,-0.6) node {$7$};

			\draw (10.6, 0.6) -- (12.6,0.6); \draw (10.6, 1) -- (12.6,1);\draw (10.6, 0.2) -- (11.8,0.2); 
			\draw (10.6, 0.2) -- (10.6,1); \draw (11, 0.2) -- (11,1);  \draw (11.4, 0.2) -- (11.4,1); \draw (11.8, 0.2) -- (11.8,1); \draw (12.2, 0.6) -- (12.2,1);	\draw (12.6, 0.6) -- (12.6,1);  	  
			\draw (10.6, -0.4) -- (12.6,-0.4);\draw (10.6, 0) -- (12.6,0);  \draw (10.6, -0.8) -- (11.8,-0.8);
			\draw (10.6, -0.8) -- (10.6,0); \draw (11, -0.8) -- (11,0); 	\draw (11.4, -0.8) -- (11.4,0);  \draw (11.8, -0.8) -- (11.8,0);\draw (12.2, -0.4) -- (12.2,0); \draw (12.6, -0.4) -- (12.6,0);
			
			\draw[black] (10.8,0.8) node {$1$};  \draw[black] (11.2,0.8) node {$3$}; \draw[black] (11.6,0.8) node {$5$}; \draw[black] (12,0.8) node {$6$}; \draw[black] (12.4,0.8) node {$7$};
			\draw[black] (10.8,0.4) node {$2$};  \draw[black] (11.2,0.4) node {$4$}; \draw[black] (11.6,0.4) node {$8$};
			\draw[black] (10.8,-0.2) node {$1$}; \draw[black] (11.2,-0.2) node {$2$}; \draw[black] (11.6,-0.2) node {$5$}; \draw[black] (12,-0.2) node {$6$}; \draw[black] (12.4,-0.2) node {$8$};
			\draw[black] (10.8,-0.6) node {$3$}; \draw[black] (11.2,-0.6) node {$4$};\draw[black] (11.6,-0.6) node {$7$};
			
			\draw[black] (13.2,0.6) node {$=P$};
			\draw[black] (13.2,-0.4) node {$=Q$};  	 	
			\end{tikzpicture}
			\caption{Construction of $(P,Q)=\RSK(\pi)$, with $\pi=24135867$.}\label{fig:rsk}
		\end{center}
	\end{figure}
	
	\hspace{-0.58cm}In \cite{ElizP}, Elizalde and Pak presented a matching algorithm  which matches some non-excedance values with excedance values of a given permutation. We remark that when we change the output of this algorithm, we directly get the second rows of P and Q. In fact, the modified matching algorithm that we present here  matches non-excedances with excedance values.

	\textbf{Matching algorithm:}
	
	INPUTS:	excedances $e_1\!<\!\ldots\!<e_k$ and non-excedances $a_1\!<\!\ldots\!<\!a_{n-k}$ of $\sigma$
	
	OUTPUT: List of matched excedances $\mathcal{M}$
	
	BEGIN
	
	\hspace{0.5cm}Let  $p := 1$; $q := 1$ and $\mathcal{M}:=\{\}$
	
	\hspace{0.5cm}REPEAT UNTIL $p>k$ OR $q>n-k$.
	
	\hspace{1cm}IF $e_p>a_q$ THEN $q:=q+1$;
	
	\hspace{1cm}ELSE IF $\sigma(e_p)<\sigma(a_q)$ THEN  $p := p + 1$; 
	
	\hspace{2.1cm}ELSE $\mathcal{M}:=\mathcal{M}\cup \{(\sigma(e_p),a_q)\}$; $p := p + 1$; $q := q + 1$;
	
	END	\\

	A given permutation $\sigma$ is \textit{bi-increasing} if its excedance and non-excedance values are both increasing. That means $\sigma(i_1)\ldots \sigma(i_k)$ and $\sigma(j_1)\ldots \sigma(j_{n-k})$ are both increasing subsequences, where $\{i_1,\ldots , i_k\}$  and $\{j_1,\ldots , j_{n-k}\}$ are respectively excedances and non-excedances of $\sigma$ in increasing order. By Reifegerste \cite{ARef}, all 321-avoiding permutations are bi-increasing. In this work, we sometime refer to this known property.
	
	Let us consider $\sigma \in S_n(321)$ and let $\mathcal{M}=\{(E_1,a_1),\ldots,(E_l,a_l)\}$  be the output of the matching algorithm   with $l\leq n$ and suppose that $(P,Q)=\RSK(\sigma)$. Since $\sigma$ is bi-increasing, then $E_1$, \ldots, $E_l$ are the excedance values of $\sigma$ that are bumped respectively by the non-excedance values $\sigma(a_1)$, \ldots, $\sigma(a_l)$ when applying the $\RSK$ algorithm. In other words, the second rows of $P$ and $Q$ are respectively $[E_1,\ldots, E_l]$  and $[a_1,\ldots, a_l]$. So we have the following remark.
	\begin{remark}\label{rem21} 
		{\rm The bumped excedance values by the $\RSK$ algorithm  and  the matched excedance values by the matching algorithm  on a 321-avoiding permutation are the same.}
	\end{remark}	
	For example, when we apply the matching algorithm with the permutation $\pi=24135867 \in S_8(321)$ of the above example, we get as output  $\mathcal{M}=\{(2,\textcolor{gray}{3}), (4,\textcolor{gray}{4}), (8,\textcolor{gray}{7})\}$. If  $(P,Q)=\RSK(\pi)$, the second rows of $P$ and $Q$ are respectively $[2,4,8]$ and $[3,4,7]$. So, we can therefore deduce the first rows of $P$ and $Q$ that are respectively $[1,3,5,6,7]$ and $[1,2,5,6,8]$.

	The second correspondence is a simple transformation of the pair of standard Young tableaux $(P,Q)$, result of $\RSK$, into Dyck path $D=\Psi(\sigma)$. The first half of the Dyck path $D$ is obtained from $P$ by adjoining, for $i$ from $1$ to $n$, an up-step if $i$ is in the first row of $P$, and a down-step if $i$ is in the second row. The second half of $D$ is obtained from $Q$ by adjoining, for $j$ from $n$ down to $1$, an up-step if $j$ is in the second row of $Q$ and a down-step if $j$ is in the first row. If we denote respectively by $D^{(L)}$  and $D^{(R)}$ the left and right half sub-paths of $D$ produced respectively by the tableaux $P$ and $Q$, then we can write $\Psi(\sigma)=D^{(L)}D^{(R)}$. 
	
	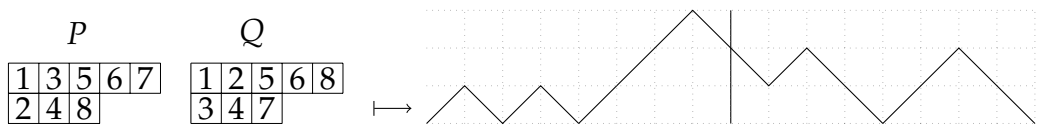
\begin{figure}[h]
		\begin{center}
			\begin{tikzpicture}
			\draw[black] (0.9,1.2) node {$P$}; \draw[black] (3.2,1.2) node {$Q$};
			\draw[step=0.4cm, black, very thin] (0,0.4) grid (2,0.8); \draw[step=0.4cm, black, very thin] (2.4,0.4) grid (4.4,0.8);
			\draw[step=0.4cm, black, very thin] (0, 0) grid (1.2,0.4);     \draw[step=0.4cm, black, very thin] (2.4, 0)  grid (3.6,0.4);
			\draw (2.4,0)-- (2.4,0.8);	
			\draw[black] (0.2,0.6) node {$1$};	\draw[black] (0.6,0.6) node {$3$};  \draw[black] (1,0.6) node {$5$}; \draw[black] (1.4,0.6) node {$6$}; \draw[black] (1.8,0.6) node {$7$};
			\draw[black] (0.2,0.2) node {$2$}; 	\draw[black] (0.6,0.2) node {$4$}; \draw[black] (1,0.2) node {$8$};
			
			\draw[black] (2.6,0.6) node {$1$};	\draw[black] (3,0.6) node {$2$}; \draw[black] (3.4,0.6) node {$5$}; \draw[black] (3.8,0.6) node {$6$}; \draw[black] (4.2,0.6) node {$8$};
			\draw[black] (2.6,0.2) node {$3$}; 	\draw[black] (3,0.2) node {$4$}; \draw[black] (3.4,0.2) node {$7$};
			
			\draw[|->,thin, black] (4.8, 0.2)--(5.3,0.2);
			\draw[-,thin, black] (9.5, 0)--(9.5,1.5);
			\draw[step=0.5cm, gray, very thin,dotted] (5.5, 0) grid (13.5,1.5);
			\draw[black] (5.5, 0)--(6, 0.5)--(6.5, 0)--(7, 0.5)--(7.5, 0);
			\draw[black](7.5, 0)--(8, 0.5);
			\draw[black](8, 0.5)--(8.5, 1)--(9, 1.5)--(9.5, 1)--(10, 0.5)--(10.5, 1)--(11, 0.5);
			\draw[black](11, 0.5)--(11.5, 0);
			\draw[black](11.5, 0)--(12, 0.5)--(12.5, 1)--(13, 0.5)--(13.5, 0);
			\end{tikzpicture}
			\caption{The corresponding Dyck path from $(P,Q)={\rm RSK}(24135867)$.}
		\end{center}
	\end{figure}
	
	From this given definition of $\Psi$, we have the following obvious proposition. 
	\begin{proposition}\label{prop22}
		For any Dyck path $D$, we have $|D^{(L)}|_d=|D^{(R)}|_u$.
	\end{proposition}
	\begin{proof}
		Let us consider a Dyck path $D$. There exists a permutation $\sigma \in S(321)$ such that $D=\Psi(\sigma)$. If $(P,Q)=\RSK(\sigma)$, then $|D^{(L)}|_d$ and $|D^{(R)}|_u$ are respectively equal to the sizes of the second rows of $P$ and $Q$. So, $|D^{(L)}|_d$ and $|D^{(R)}|_u$ are equal.
	\end{proof}
	\subsection{The bijection $\Phi^{-1}:\mathcal{D}_{n}\rightarrow S_n(132)$} \label{sec23}
	As mentioned in \cite{ElizP},  the bijection $\Phi$ is essentially the same bijection between $S_n(132)$ and $\D_n$ given by Krattenthaler \cite{Krat}, up to reflection of the path over a vertical line. Here, we will present the bijection $\Phi^{-1}$ in a slightly different way.
	
	Starting from a Dyck path  $D\in \mathcal{D}_{n}$, we will construct the corresponding permutation  by the following procedure.
	From left to right, number the up-steps of $D$ from $n$ down to $1$ and the down-steps from $1$ up to $n$.  Then, we have  $\Phi^{-1}(D)(n+1-i)=j$ if and only if tunnel from up-step numbered $n+1-i$ (i.e.~the $i$-th up-step) is ending to down-step numbered $j$ (i.e.~the $j$-th down-step). By this way, it is not hard to show that the map $\Phi^{-1}$ is a well defined biijection. See Fig. \ref{fig:form1} for graphical illustration.
	
	\begin{figure}[h]
		\begin{center}
			\begin{tikzpicture}
			\draw[step=0.5cm, gray, very thin,dotted] (0, 0) grid (8,1.5);
			\draw [-, black](4, -0.2)--(4, 1.5);
			\draw[black] (0, 0)--(0.5, 0.5)--(1, 0)--(1.5, 0.5)--(2, 0)--(2.5, 0.5)--(3, 1)--(3.5, 1.5)--(4, 1)--(4.5, 0.5)--(5, 1)--(5.5, 0.5)--(6, 0)--(6.5, 0.5)--(7, 1)--(7.5, 0.5)--(8, 0);
			\draw[black] (0.25,0.4) node {$8$};
			\draw[black] (1.25,0.4) node {$7$};
			\draw[black] (2.25,0.4) node {$6$}; 
			\draw[black] (2.75,0.9) node {$5$};
			\draw[black] (3.25,1.4) node {$4$}; 
			\draw[black] (4.75,0.9) node {$3$};
			\draw[black] (6.25,0.4) node {$2$}; 
			\draw[black] (6.75,0.9) node {$1$};   
			
			\draw[gray] (0.75,0.4) node {$1$};
			\draw[gray] (1.75,0.4) node {$2$};
			\draw[gray] (3.75,1.35) node{$3$};
			\draw[gray] (4.25,0.9) node {$4$};
			\draw[gray] (5.25,0.9) node {$5$}; 
			\draw[gray] (5.75,0.4) node {$6$};
			\draw[gray] (7.25,0.9) node {$7$}; 
			\draw[gray] (7.75,0.4) node {$8$}; 
			
			\draw [->,dashed, gray](0,0)--(0.95,0);
			\draw [->,dashed, gray](1,0)--(1.95,0);
			\draw [->,dashed, gray](3,1)--(3.95,1);
			\draw [->,dashed, gray](2.5,0.5)--(4.45,0.5);
			\draw [->,dashed, gray](2,0)--(5.95,0);
			\draw [->,dashed, gray](4.5,0.5)--(5.45,0.5);
			\draw [->,dashed, gray](6.5,0.5)--(7.45,0.5);
			\draw [->,dashed, gray](6,0)--(7.95,0);
			\draw[|->,thin, gray] (8.5, 0.75)--(9,0.75);
			\draw[black] (11.5,0.75) node {$\begin{pmatrix}
				\textcolor{black}{1}& \textcolor{black}{2}& \textcolor{black}{3}& \textcolor{black}{4}& \textcolor{black}{5}& \textcolor{black}{6}& \textcolor{black}{7}& \textcolor{black}{8} \\
				\textcolor{gray}{7}& \textcolor{gray}{8}& \textcolor{gray}{5}& \textcolor{gray}{3}& \textcolor{gray}{4}& \textcolor{gray}{6}& \textcolor{gray}{2}& \textcolor{gray}{1}
				\end{pmatrix}$};
			\end{tikzpicture}
			\caption{Corresponding 132-avoiding permutation from a Dyck path.}\label{fig:form1}
		\end{center}
	\end{figure}
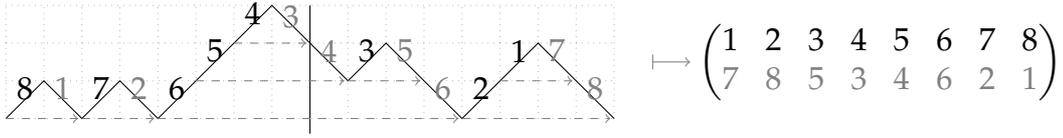
	
	Let us consider $\sigma \in S_n$ and two integers $a$ and $b$ satisfying $1\leq a\leq b \leq n$. We  denote by $\sigma(a\ldots b):=\sigma(a)\sigma(a+1)\ldots\sigma(b)$ the contiguous subsequence of $\sigma$ from its $a$-th to  $b$-th letter and  for any operator $* \in  \{<, \leq, \neq, \geq, >\}$ and a given number $x$, we can write $\sigma(a\ldots b)* x$ if only if $\sigma(i)* x$ for all $i\in [a;b]$.  Example: if we consider $\pi=6413275 \in S_7$, then we have $\pi(2\ldots 5)=4132\leq 4$. Our first observation of the bijection $\Phi$ described by the above procedure leads to the following proposition.
	\begin{proposition}\label{prop23}
		Let us assume that $D\in \D_n$ and $j=|D^{(R)}|_u+1$. The permutation $\sigma=\Phi^{-1}(D)$ satisfies the following properties:
		\begin{itemize}
			\setlength\itemsep{-0.3em}
			\item[{\rm (i)}] if $j\geq 2$, then we have $\sigma^{-1}(1\ldots j-1)\geq j\leq \sigma(1\ldots j-1)$,
			\item[{\rm (ii)}] for all $i\geq j$, if $\sigma(i)>i$,  then we have $\sigma^{-1}(i)<i$.
		\end{itemize}
	\end{proposition}
	\begin{proof}
		Let us consider $D\in \D_n$ and suppose that $j=|D^{(R)}|_u+1\geq 2$. We range in the following table all of assigned numbers to up-steps and down-steps of $D$ for getting $\sigma=\Phi^{-1}(D)$.
		\begin{table}[h]
			\begin{center}
				\begin{tabular}{|c|c|c|}
					\hline
					Sub-path	& $D^{(L)}$  & $D^{(R)}$\\
					\hline
					For up-steps & $n,\ldots,j+1,j$&  $j-1,\ldots, 2,1$\\
					\hline
					For down-steps &$1,2,\ldots,j-1$ &$j,j+1,\ldots,n$\\
					\hline
				\end{tabular}
				\caption{Assigned numbers to up-steps and down-steps of $D$.}
				\label{table:tabx0}
			\end{center}
		\end{table}
		When looking at the second column of Table \ref{table:tabx0}, we get $\sigma(1\ldots j-1)\geq j$. Similarly, when looking at the first column, we also get $\sigma^{-1}(1\ldots j-1)\geq j$. So, we easily obtain the proof of the first property.
		
		Now let us consider an integer $i$ such that that $\sigma(i)>i\geq j$. If we return to the Dyck path $D$, then the tunnel which matches the up-step numbered $i$ with the down-step numbered $\sigma(i)$ decomposes the Dyck path $D$ as $D = \ldots uBd \ldots$,  where $u$ is the $(n+1-i)$-th up-step of $D$, $d$ is the $\sigma(i)$-th down-step of $D$, $B$ is a sub-Dyck path of $D$  which contains at least the down-steps numbered  $j,j+1,\ldots, \sigma(i)-1$ and its up-steps are obviously numbered by numbers less than $i$. This implies that we must have $\sigma^{-1}(j\ldots \sigma(i)-1)<i$. Since $i\in \{j,j+1,\ldots, \sigma(i)-1\}$, so we get $\sigma^{-1}(i)<i$. This ends the proof of the second property of our proposition.		
	\end{proof}
	Let us just end this section with the following obvious remark which implies that $\Phi^{-1}$ exchanges left and centered tunnels of a Dyck path and non-excedances of the corresponding 132-avoiding permutation.
	\begin{remark}\label{rem23}
		{\rm Tunnel of Dyck path $D\in \D_n$ which matches the up-step numbered $n+1-i$ with the down-step numbered $j$ is left or centered  if and only if  $n+1-i\geq j=\Phi^{-1}(D)(n+1-i)$}.
	\end{remark}
	\subsection{A direct formulation of $\Theta$}\label{sec24}
	As seen in the previous sub-sections, the original definition of $\Theta$ uses Young tableaux and Dyck paths as auxiliary combinatorial objects.  In this section, we propose a recursive formula for the bijection $\Theta$ which does not use these auxiliary objects. For that, we introduce some needed notations. Given a permutation $\sigma \in S_n$, 
	\begin{itemize}
		\setlength\itemsep{-0.3em}
		\item[-]  $\sigma^{+a}$ denotes the obtained permutation from  $\sigma$ by adding $a$ to each of its number. Example: $312^{+2}=534$.
		\item[-] $\sigma^{a \rtimes b}$ denotes the obtained permutation  from $\sigma$ by adding $b$ all of its numbers greater or equal to $a$. Example: $4132^{3\rtimes 2}=\textcolor{gray}{6}1\textcolor{gray}{5}2$.
		\item[-]  if $a$ and $b$ satisfy $1\leq a,b\leq |\sigma|+1$, then $\sigma^{(a,b)}$  denotes  the obtained permutation from $\sigma$ by inserting the number $b$ at the $a$-th position of $\sigma$. More precisely, we have $\sigma^{(a,b)}=\sigma^{b\rtimes 1}(1\ldots a-1). b. \sigma^{b\rtimes 1}(a\ldots |\sigma|)$.
		Example: $3142^{(2,\textcolor{gray}{3})}=4\textcolor{gray}{3}152$. 
	\end{itemize}
	\begin{proposition} \label{prop:prop1}
		If $\sigma \in S_n^k(321):=\{\sigma \in S_n(321)|\ \sigma(n)=k\}$ and $\pi=red[\sigma(1)\ldots \sigma(n-1)]$ (i.e.~$\sigma=\pi^{(n,k)}$),  then we have $\Theta(\sigma)=\Theta(\pi)^{(n-k+j,j)}$, where $j-1$ is the number of matched excedance values less than $k$ of $\pi$. Furthermore, $n-k+j$ is the minimum of non-excedance of $\Theta(\sigma)$. 
	\end{proposition}
	\begin{proof}
		Let us assume that $\sigma \in S_n^k(321)$ and $\pi=red[\sigma(1\ldots n-1)]$. Since $\sigma=\pi^{(n,k)}$, the pair $(P,Q)=\RSK(\sigma)$ is obtained from $(P',Q')=\RSK(\sigma(1\ldots n-1))$ by inserting  $(n,k)$. Following the logical of insertion of $(n,k)$ in $(P',Q')$ to get $(P,Q)$, we observe that $\Psi(\sigma)$ can also be obtained from $\Psi(\pi)$. For better understanding, we reason graphically and we distinguish three cases according to the values of $k$. In all cases, we denote by $(i-1)$ (resp $(j-1))$ the position of the rightmost number less than $k$ in the first row (resp second row) of $P'$. In other words, the number at the $i$-th  (resp $j$-th) column of the first row (resp second row) of $P'$ is greater than $k$ if it exists.  Notice that $j-1$ can be interpreted as the number of bumped excedance values less than $k$ by the $\RSK$ algorithm and following Remark \ref{rem21} it is also the  number of the matched excedance values less than $k$ by the matching algorithm.
		
		Let us first suppose that $k=n$. To get $(P,Q)$, we just add $n$ at the end  of the first rows of $P'$ and $Q'$ (see Table \ref{fig:proof1}). Consequently, when we translate $(P,Q)$ into Dyck path, we get $\Psi(\sigma)=\Psi(\pi)^{(L)}. ud . \Psi(\pi)^{(R)}$, where $\Psi(\sigma)^{(L)}\Psi(\sigma)^{(R)}=\Psi(\sigma)$. So, when we number the steps of $\Psi(\sigma)$ to get $\Theta(\sigma)$ according to the procedure described in section \ref{sec23}, those of newly added up-step $u$ and down-step $d$ are respectively $n+1-i$ and $j$. 
		
		\begin{table}[h]
			\begin{center}
				\begin{tabular}{|cc|c|}	
					\hline	
					\begin{tikzpicture}
					\draw[black] (-0.25,0.25) node[scale=0.6pt] {P'=};
					\draw[black] (1,1) node[scale=0.7pt] {Before insertion of $(n,n)$};
					\draw (0, 0.5) -- (2,0.5);\draw (0, 0.25) -- (2,0.25); \draw (0, 0) -- (1.25,0); 
					\draw (0, 0) -- (0,0.5);   \draw (2,0.25) -- (2,0.5); \draw (1.25, 0) -- (1.25,0.25); 
					\draw[black] (1,0.35) node[scale=0.6pt] {$\ldots$};	\draw[black] (1.4,0.75) node[scale=0.6pt] {$i-1$-th}; \draw[->,thin, black] (1.75,0.7)--(1.95,0.5);
					\draw[black] (0.4,0.1) node[scale=0.6pt] {$\ldots$};
					\draw[black] (0.8,-0.3) node[scale=0.6pt] {$(j-1)$-th}; \draw[->,thin, black] (0.8,-0.25)--(1.2,0); 	
					\end{tikzpicture}
					&\hspace{1cm} &
					\begin{tikzpicture}			
					\draw[black] (-0.25,0.25) node[scale=0.6pt] {P=};
					\draw[black] (1,1) node[scale=0.7pt] {After insertion of $(n,n)$};
					\draw (0, 0.5) -- (2,0.5);\draw (0, 0.25) -- (2,0.25); \draw (0, 0) -- (1.25,0); 
					\draw (0, 0) -- (0,0.5); \draw (1.25, 0) -- (1.25,0.25); \draw (1.65,0.25) -- (1.65,0.5); \draw (2,0.25) -- (2,0.5); 
					\draw[black] (1,0.35) node[scale=0.6pt] {$\ldots$};	\draw[black] (0.4,0.1) node[scale=0.6pt] {$\ldots$};
					\draw[black] (0.8,-0.3) node[scale=0.6pt] {$(j-1)$-th}; \draw[->,thin, black] (0.8,-0.25)--(1.2,0);
					\draw[black] (1.7,0.75) node[scale=0.6pt] {$i$-th}; \draw[->,thin, black] (1.75,0.7)--(1.85,0.5);	\draw[red] (1.85,0.35) node[scale=0.6pt] {n};

					\draw[black] (2.5,0.25) node[scale=0.6pt] {Q=};
					\draw (2.75, 0.5) -- (4.75,0.5);\draw (2.75, 0.25) -- (4.75,0.25); \draw (2.75, 0) -- (4,0);
					\draw (2.75, 0) -- (2.75,0.5);  \draw (4, 0) -- (4,0.25); \draw (4.4,0.25) -- (4.4,0.5); \draw (4.75,0.25) -- (4.75,0.5);
					\draw[black] (3.75,0.35) node[scale=0.6pt] {$\ldots$};	\draw[black] (3.15,0.1) node[scale=0.6pt] {$\ldots$};
					\draw[red] (4.6,0.35) node[scale=0.6pt] {n};
					\end{tikzpicture}
					\\ 
					\begin{tikzpicture}
					\draw[black] (0.25,0.75) node[scale=0.6pt] {$\Psi(\pi)$=};
					\draw[dotted,black] (1, 0.75)--(1.5, 0.75);
					\draw[black] (1.5, 0)--(1.5, 1.5);  
					\draw[dotted,black] (1.5, 0.75)--(2, 0.75);
					\end{tikzpicture}
					& &
					
					\begin{tikzpicture}
					\draw[black] (0.25,0.75) node[scale=0.6pt] {$\Psi(\sigma)$=};
					\draw[dotted,black] (1, 0.75)--(1.5, 0.75);
					\draw[black] (1.3, 0.95) node[scale=0.6pt] {n+1-i};
					\draw[red] (1.5, 0.75)--(1.75, 1);
					\draw[red] (1.75, 1)--(2, 0.75);  \draw[black] (2, 0.9) node[scale=0.6pt] {j}; 
					\draw[black] (1.75, 0)--(1.75, 1.5);  
					\draw[dotted,black] (2, 0.75)--(2.5, 0.75); 
					\end{tikzpicture}
					\\		\hline
				\end{tabular}
				\caption{Insertion of $(n,n)$.}\label{fig:proof1}
			\end{center}
		\end{table}
		
		Suppose now that $k<n$ but it is also greater than all numbers in the first row of $P'$. In this case, when inserting $k$ in $P'$, bumping does not occur but there exists a part $A$ in the second row of $P'$ such that its elements are all greater than $k$. As we can see in Table \ref{fig:proof2}, $\Psi(\pi)^{(L)}$ is ending with a sequence of down-steps produced by $A$.  The corresponding Dyck path $\Psi(\sigma)$ of $(P,Q)$ can be obtained  from $\Psi(\pi)$. Indeed,  we have $\Psi(\sigma)^{(R)}=d.\Psi(\pi)^{(R)}$ and we obtain $\Psi(\sigma)^{(L)}$ from $\Psi(\pi)^{(L)}$ by inserting a new up-step (produced by $k$) just before the $j$-th down-step (produced by the minimum of $A$). So, when numbering the steps of $\Psi(\sigma)$ to get $\Theta(\sigma)$,  that of the new added up-step is $n+1-i$ which is just before the down-step numbered $j$. 
		
		\begin{table}[h]
			\begin{center}
				\begin{tabular}{|cc|c|}	
					\hline	
					\begin{tikzpicture}
					\draw[black] (-0.25,0.25) node[scale=0.6pt] {P'=};
					\draw[black] (1,1) node[scale=0.7pt] {Before insertion of $(n,k)$};
					\draw (0, 0.5) -- (2,0.5);\draw (0, 0.25) -- (2,0.25); \draw (0, 0) -- (1.25,0); 
					\draw (0, 0) -- (0,0.5); \draw (0.75, 0) -- (0.75,0.25); \draw (1.25, 0) -- (1.25,0.25); \draw (2,0.25) -- (2,0.5); 
					\draw[black] (1,0.35) node[scale=0.6pt] {$\ldots$};	\draw[black] (1.4,0.75) node[scale=0.6pt] {$i-1$-th}; \draw[->,thin, black] (1.75,0.7)--(1.95,0.5);				
					\draw[black] (0.4,0.1) node[scale=0.6pt] {$\ldots$};
					\draw[black] (0.4,-0.3) node[scale=0.5pt] {$j-1$-th}; \draw[->,thin, black] (0.5,-0.25)--(0.65,0); \draw[black] (1,0.1) node[scale=0.6pt] {$A$};	
					\end{tikzpicture}
					&\hspace{1cm} &
					\begin{tikzpicture}			
					\draw[black] (-0.25,0.25) node[scale=0.6pt] {P=};
					\draw[black] (1,1) node[scale=0.7pt] {After insertion of $(n,k)$};
					\draw (0, 0.5) -- (2,0.5);\draw (0, 0.25) -- (2,0.25); \draw (0, 0) -- (1.25,0); 
					\draw (0, 0) -- (0,0.5); \draw (0.75, 0) -- (0.75,0.25); \draw (1.25, 0) -- (1.25,0.25); \draw (1.65,0.25) -- (1.65,0.5); \draw (2,0.25) -- (2,0.5);
					\draw[black] (1,0.35) node[scale=0.6pt] {$\ldots$};	\draw[black] (0.4,0.1) node[scale=0.6pt] {$\ldots$};
					\draw[black] (0.4,-0.3) node[scale=0.5pt] {$j-1$-th}; \draw[->,thin, black] (0.5,-0.25)--(0.65,0); \draw[black] (1,0.1) node[scale=0.6pt] {$A$};
					\draw[black] (1.7,0.75) node[scale=0.6pt] {$i$-th}; \draw[->,thin, black] (1.75,0.7)--(1.85,0.5);	\draw[red] (1.85,0.35) node[scale=0.6pt] {k};

					\draw[black] (2.5,0.25) node[scale=0.6pt] {Q=};
					\draw (2.75, 0.5) -- (4.75,0.5);\draw (2.75, 0.25) -- (4.75,0.25); \draw (2.75, 0) -- (4,0);
					\draw (2.75, 0) -- (2.75,0.5);  \draw (4, 0) -- (4,0.25); \draw (4.4,0.25) -- (4.4,0.5); \draw (4.75,0.25) -- (4.75,0.5);
					\draw[black] (3.75,0.35) node[scale=0.6pt] {$\ldots$};	\draw[black] (3.15,0.1) node[scale=0.6pt] {$\ldots$};
					\draw[red] (4.6,0.35) node[scale=0.6pt] {n};
					\end{tikzpicture}
					\\ 
					\begin{tikzpicture}
					\draw[black] (0.25,0.75) node[scale=0.6pt] {$\Psi(\pi)$=};
					\draw[dotted,black] (1, 1.5)--(1.5, 1.5);
					\draw[black] (1.82,1.4) node[scale=0.6pt] {j}; \draw[black] (1.5, 1.5)--(1.75, 1.25);
					\draw[dotted,black] (1.75, 1.25)--(2, 1); 
					\draw[black] (2, 1)--(2.25, 0.75);
					\draw[black] (2.25, 0)--(2.25, 2);  
					\draw[dotted,black] (2.25, 0.75)--(2.75, 0.75);
					\end{tikzpicture}
					& &
					
					\begin{tikzpicture}
					\draw[black] (0.25,0.75) node[scale=0.6pt] {$\Psi(\sigma)$=};
					\draw[dotted,black] (1, 1.5)--(1.5, 1.5);
					\draw[black] (1.25, 1.75) node[scale=0.6pt] {n+1-i};\draw[red] (1.5, 1.5)--(1.75, 1.75);
					\draw[black] (1.75, 1.75)--(2, 1.5);  \draw[black] (2, 1.75) node[scale=0.6pt] {j}; 
					\draw[dotted,black] (2, 1.5)--(2.25, 1.25);
					\draw[black] (2.25, 1.25)--(2.5, 1);
					\draw[black] (2.5, 0)--(2.5, 2);  
					\draw[red] (2.5,1)--(2.75, 0.75);
					\draw[dotted,black] (2.75, 0.75)--(3.25, 0.75); 
					\end{tikzpicture}
					\\	\hline
				\end{tabular}
				\caption{Insertion of $(n,k)$ with $k<n$ and bumping does not occur.}\label{fig:proof2}
			\end{center}
		\end{table}
		
		The last case that we discuss here is the case where bumping occur, i.e.~$k<n$ and there is at least an element of the first row of $P'$ which is greater than $k$. Let us denote by $x$ the leftmost of such element. It is therefore at the $i$-th column of $P'$. In the second row of $P'$, we still denote by $A$ the part of elements greater than $k$ which begins at the $j$-th comlumn. Notice that any elements of $A$ is also less that $x$, i.e.~$A\subseteq\{k+1,k+2,\ldots, x-1\}$. That is why, as we can see in Table \ref{fig:proof3}, the path $\Psi(\pi)^{(L)}$ is ending with a sequence of down-steps produced by $A$ followed by a sequence of up-steps produced $x$ and the numbers in its right.  When inserting $k$ in $P'$, $k$ takes the place of $x$ (i.e.~$x$ is bumped by $k$)  and $x$ creates a new cell at the end of the second row. Similarly to the previous cases, the deduction of  $\Psi(\sigma)$ from $\Psi(\pi)$ simply follows the logical of insertion of $(n,k)$ in $(P',Q')$ to get $(P,Q)$. In fact, we have $\Psi(\sigma)^{(R)}=u.\Psi(\pi)^{(R)}$ and $\Psi(\sigma)^{(R)}$ is obtained from $\Psi(\pi)^{(R)}$ by replacing the $i$-th up-step (the up-step produced by $x$) by a down-step because $x$ becomes an element of the second row, then inserting a new up-step (the up-step produced by $k$) just before the $j$-th down-step.

		\begin{table}[h]
			\begin{center}
				\begin{tabular}{|cc|c|}	
					\hline	
					\begin{tikzpicture}
					\draw[black] (-0.25,0.25) node[scale=0.6pt] {P'=};
					\draw[black] (1,1) node[scale=0.7pt] {Before insertion of $(n,k)$};
					\draw (0, 0.5) -- (2.5,0.5);\draw (0, 0.25) -- (2.5,0.25); \draw (0, 0) -- (1.25,0); 
					\draw (0, 0) -- (0,0.5); \draw (0.75, 0) -- (0.75,0.25); \draw (1.25, 0) -- (1.25,0.25);\draw (1.75, 0.25) -- (1.75,0.5); \draw (2, 0.25) -- (2,0.5); \draw (2.5,0.25) -- (2.5,0.5); 
					\draw[black] (1,0.35) node[scale=0.6pt] {$\ldots$};	\draw[black] (0.4,0.1) node[scale=0.6pt] {$\ldots$}; \draw[black] (2.25,0.35) node[scale=0.6pt] {$\ldots$};
					\draw[black] (0.4,-0.3) node[scale=0.5pt] {$j-1$-th}; \draw[->,thin, black] (0.5,-0.25)--(0.65,0); \draw[black] (1,0.1) node[scale=0.6pt] {$A$}; 
					\draw[black] (1.7,0.75) node[scale=0.6pt] {$i$-th}; \draw[->,thin, black] (1.75,0.7)--(1.85,0.5); \draw[black] (1.85,0.35) node[scale=0.6pt] {$x$};	
					\end{tikzpicture}
					&\hspace{1cm} &
					\begin{tikzpicture}			
					\draw[black] (-0.25,0.25) node[scale=0.6pt] {P=};
					\draw[black] (1.5,1) node[scale=0.7pt] {After insertion of $(n,k)$};
					\draw (0, 0.5) -- (2.5,0.5);\draw (0, 0.25) -- (2.5,0.25); \draw (0, 0) -- (1.5,0); 
					\draw (0, 0) -- (0,0.5); \draw (0.75, 0) -- (0.75,0.25); \draw (1.25, 0) -- (1.25,0.25);\draw (1.75, 0.25) -- (1.75,0.5); \draw (2, 0.25) -- (2,0.5); \draw (2.5,0.25) -- (2.5,0.5);
					\draw (1.5,0) -- (1.5,0.25); 
					\draw[black] (1,0.35) node[scale=0.6pt] {$\ldots$};	\draw[black] (0.4,0.1) node[scale=0.6pt] {$\ldots$}; \draw[black] (2.25,0.35) node[scale=0.6pt] {$\ldots$};
					\draw[black] (0.4,-0.3) node[scale=0.5pt] {$j-1$-th}; \draw[->,thin, black] (0.5,-0.25)--(0.65,0); \draw[black] (1,0.1) node[scale=0.6pt] {$A$};
					\draw[black] (1.7,0.75) node[scale=0.6pt] {$i$-th}; \draw[->,thin, black] (1.75,0.7)--(1.85,0.5);	\draw[red] (1.85,0.35) node[scale=0.6pt] {k};
					\draw[gray] (1.35,0.1) node[scale=0.6pt] {$x$};
					
					\draw[black] (3,0.25) node[scale=0.6pt] {Q=};
					\draw (3.25, 0.5) -- (5.25,0.5);\draw (3.25, 0.25) -- (5.25,0.25); \draw (3.25, 0) -- (4.5,0);
					\draw (3.25, 0) -- (3.25,0.5);  \draw (4.5, 0) -- (4.5,0.25); \draw (5.25,0.25) -- (5.25,0.5); \draw (4.5,0.) -- (4.5,0.25); \draw (4.25,0.) -- (4.25,0.25);
					\draw[black] (4.25,0.35) node[scale=0.6pt] {$\ldots$};	\draw[black] (3.75,0.1) node[scale=0.6pt] {$\ldots$};
					\draw[red] (4.4,0.1) node[ scale=0.6pt] {$n$};
					\end{tikzpicture}
					\\
					\begin{tikzpicture}
					\draw[black] (0.25,0.75) node[scale=0.6pt] {$\Psi(\pi)$=};
					\draw[dotted,black] (1, 1.5)--(1.5, 1.5);
					\draw[black] (1.8,1.45) node[scale=0.6pt] {j}; \draw[black] (1.5, 1.5)--(1.75, 1.25);
					\draw[dotted,black] (1.75, 1.25)--(2, 1);
					\draw[black] (2, 1)--(2.25, 0.75);
					\draw[black] (2.25, 0.75)--(2.5, 1);
					\draw[dotted] (2.5, 1)--(2.75, 1.25); \draw[black] (2.5, 0.9) node[scale=0.5pt] {n+1-i}; 
					\draw[black] (2.75, 0)--(2.75, 2); 
					\draw[dotted,black] (2.75,1.25)--(3.25,1.25);
					\end{tikzpicture}
					& &
					\begin{tikzpicture}
					\draw[black] (0.25,0.75) node[scale=0.6pt] {$\Psi(\sigma)$=};
					\draw[dotted,black] (1, 1.5)--(1.5, 1.5);
					\draw[black] (1.25, 1.75) node[scale=0.6pt] {n+1-i};\draw[red] (1.5, 1.5)--(1.75, 1.75);
					\draw[black] (1.75, 1.75)--(2, 1.5);  \draw[black] (2, 1.75) node[scale=0.6pt] {j}; 
					\draw[dotted,black] (2, 1.5)--(2.25, 1.25);
					\draw[black] (2.25, 1.25)--(2.5, 1);
					\draw[gray] (2.5,1)--(2.75, 0.75);
					\draw[dotted] (2.75, 0.75)--(3, 1); 
					\draw[red] (3,1)--(3.25,1.25);
					\draw[black] (3, 0)--(3, 2);
					\draw[dotted,black] (3.25,1.25)--(3.75,1.25); 
					\end{tikzpicture}
					\\		\hline
				\end{tabular}
				\caption{Insertion of $(n,k)$ with $k<n$ and bumping  occur.}\label{fig:proof3}
			\end{center}
		\end{table}

		\hspace{-0.65cm}In all the cases discussed above, when we look at the path $\Psi(\sigma)$, we can observe three things.
		\begin{itemize}
			\setlength\itemsep{-0.3em}
			\item Firstly, we have  $\Theta(\sigma)(n+1-i)=j$.
			\item Secondly, when we remove the up-step number $n+1-i$ and the down-step number $j$ of $\Psi(\sigma)$, the remaining Dyck path is $\Psi(\pi)$. This also implies that when we remove $\Theta(\sigma)(n+1-i)$ (which is equal to $j$) in $\Theta(\sigma)$, the obtained reduction is none other than $\Theta(\pi)$. In other words, we have $\Theta(\sigma)=\Theta(\pi)^{(n+1-i,j)}$. Moreover, since $\sigma$ is bi-increasing (see \cite{ARef}), the number $i$ equals to the number of non-excedance values less than $k$ of $\sigma$ plus  the number of non-bumped excedance values of $\sigma$ less than $k$. Consequently, since the number of bumped excedance values less than $k$ is $j-1$,  then we have $i=k-(j-1)=k+1-j$. This implies that $n+1-i=n-k+j$  and we finally get $\Theta(\sigma)=\Theta(\pi)^{(n-k+j,j)}$.
			\item Thirdly, since $n-k+j\geq j$ then $n-k+j$ is a non-excedance of $\Theta(\sigma)$. Moreover, knowing that $n-k+j$ is the minimum of the numbers assigned to up-steps of $\Psi(\sigma)^{(L)}$ in which associated tunnel is left or centered, it is also the minimum of the non-excedance of $\Theta(\sigma)$ (see Remark \ref{rem23}).
		\end{itemize}
		This ends the proof of Proposition \ref{prop:prop1}.
	\end{proof}
	
	The given property of the bijection $\Theta$ in Proposition \ref{prop:prop1} is really fundamental and allows us to  define it directly without passing to Young tableaux and Dyck paths. So we have the following theorem.
	\begin{theorem} \label{thm:thm3}
		If $\sigma \in S_n(321)$ and $\pi=red(\sigma(1\ldots n-1))$, then we have
		$$\Theta(\sigma) =\begin{cases}
		\sigma & \text{ if\ \  $|\sigma|=1$}\\
		\Theta(\pi)^{(n-\sigma(n)+j,j)} & \text{ if\ \ $|\sigma|>1$}
		\end{cases},$$ 
		where $j-1$  is the number of matched excedance values less than $\sigma(n)$. To compute $\Theta^{-1}(\alpha)$ from a given permutation $\alpha\in S(132)$, we use the following relation
		$$\Theta^{-1}(\alpha)  =\begin{cases}
		\alpha & \text{ if $|\alpha|=1$}\\
		\Theta^{-1}(\beta)^{(|\alpha|,|\alpha|+\alpha(k)-k)}& \text{ if $|\alpha|>1$}
		\end{cases},$$
		where $k$ is the minimum of the non-excedance of $\alpha$ and $\beta^{(k,\alpha(k))}=\alpha$.
	\end{theorem}
	\hspace{-0.58Cm}Notice that, the number of matched excedance values is obtained by applying the matching algorithm. Example: for $\sigma=4162735\in S_7(321)$, we summarize in table \ref{table:tab2} the computation of $\Theta(\sigma)$. 
	\begin{table}[h]
		\begin{center}
			\begin{tabular}{|c|c|c|c|}
				\hline
				$l$ & $\sigma_l=red[\sigma(1)\ldots \sigma(l)]$ & $(l-\sigma(l)+j,j)$ & $\Theta(\sigma_{l-1})^{(l-\sigma(l)+j,j)}$\\
				\hline
				1& \textcolor{gray}{1} & - &\textcolor{gray}{1} \\
				2& 2\textcolor{gray}{1} & $(2,1)$&$2\textcolor{gray}{1}$\\
				3& 21\textcolor{gray}{3} & $(2,2)$ &$3\textcolor{gray}{2}1$\\
				4& 314\textcolor{gray}{2} & $(3,1)$ &$43\textcolor{gray}{1}2$\\
				5& 3142\textcolor{gray}{5} &$(3,3)$ &$54\textcolor{gray}{3}12$\\
				6& 41526\textcolor{gray}{3} & $(4,1)$ &$654\textcolor{gray}{1}23$\\
				7& $\sigma$=416273\textcolor{gray}{5} & $(4,2)$&$\Theta(\sigma)=765\textcolor{gray}{2}134$\\
				\hline
			\end{tabular}
			\caption{Recursive computation of $\Theta(4162735)$.}
			\label{table:tab2}
		\end{center}
	\end{table}
	
	If $\alpha=7652134$, we have  $\Theta^{-1}(\alpha)=\Theta^{-1}(654123)^{(7,5)}$ since $|\alpha|=7$, $k=4$ and $\alpha(k)=2$.
	
	\section{The $(\fp, \exc,\crs)$-preserving of the bijection $\Theta$}\label{sec3}
	In \cite{ElizP}, Elizalde and Pak proved that the bijection $\Theta$ preserves the number of fixed points and excedances. 
	\begin{theorem}{\rm \cite{ElizP}}\label{ElizP}
		The bijection $\Theta$ preserves the number of fixed points and excedances. It means that, for any permutation $\sigma \in S(321)$, we have $(\fp,\exc)\left( \Theta(\sigma)\right) =(\fp,\exc)(\sigma)$.
	\end{theorem}
	\hspace{-0.60cm}To prove this theorem, Elizalde and Pak showed that the bijections $\Psi$ and $\Phi$ exchange the statistics (\fp,\exc) on 321 or 132-avoiding permutations and (\ct,\rt) on Dyck paths. More precisely, we have $\Theta: (\fp,\exc) \stackrel{\Psi}{\longrightarrow} (\ct,\rt)  \stackrel{\Phi^{-1}}{\longrightarrow} (\fp,\exc).$ Using the given recursive definition of $\Theta$ in section \ref{sec24}, we will prove by induction on $n$ that it also preserves the number of crossings. For that, we need some operations and notations to be defined.
	
	Let us fix two permutations $\alpha$ and $\beta$. The \textit{direct sum}  of $\alpha$ and $\beta$ is $\alpha\oplus\beta:=\alpha.\beta^{+|\alpha|}$. Example: $\textcolor{gray}{312}\oplus 231=\textcolor{gray}{312}564$ since $231^{+3}=564$. Our aim is to define a new operation on $S(132)$ that we need to prove the \crs-preserving of the bijection $\Theta$. For that, we need the following fundamental proposition.
	\begin{proposition}\label{propo}
		Assume that $\sigma \in S_n(132)$ and denote by $T(\sigma):=
		\{i\in [n]\mid\sigma^{-1}(i)>i<\sigma(i)\}$. We have the following properties. 
		\begin{itemize}
			\setlength\itemsep{-0.3em}
			\item[(a)]  $T(\sigma)=\emptyset$ if and only if $\sigma(1)=1$.
			\item[(b)] If $j\in T(\sigma)$, then  we also have $i \in T(\sigma)$ for all $i\leq j$.
			\item[(c)] We have $|T(\sigma)|\leq \frac{n}{2}$.
			\item[(d)] If $D=\Phi(\sigma)$, then $|T(\sigma)|=|D^{(R)}|_u$.
			\item[(e)] If $k=1+|T(\sigma)|$, then we have $\sigma^{(k,k)}\in S_{n+1}(132)$.
		\end{itemize}	
	\end{proposition}
	\begin{proof}
		Let $\sigma \in S_n(132)$. Notice first that the property $(a)$ is obvious since  we have  $\sigma(1)=1$ if and only if $\sigma=12\ldots n$. Using the fact that $\sigma$ is 132-avoiding, we can easily prove by contradiction that the property $(b)$ also hold. Suppose that $t=\max T(\sigma)$. According to the property $(b)$, we have $t=|T(\sigma)|$ and so $\sigma^{-1}(1\ldots t)>t<\sigma(1\ldots t)$. This is possible only if we have $n-t\geq t$. This implies that $t\leq \frac{n}{2}$ and  ends the proof of the property $(c)$. The property $(d)$ comes from Proposition \ref{prop23}.\\
		Suppose now that $k=1+|T(\sigma)|$. Like  the property $(b)$, the property $(e)$ can be proved easily by  contradiction. Firstly, we have $\sigma^{(k,k)}=\sigma^{k\rtimes 1}(1\ldots k-1).k.\sigma^{k\rtimes 1}(k\ldots n)$ and according to the property $(b)$, we have $\sigma^{k\rtimes 1}(1\ldots k-1)=(\sigma(1\ldots k-1))^{+1}$ since $\sigma(1\ldots k-1)\geq k$. Secondly, if $\sigma^{(k,k)}\notin S_{n+1}(132)$,  then  it is $k.\sigma^{k\rtimes 1}(k\ldots n)$ which contains at least a 132-pattern because  $\sigma^{k\rtimes 1}(1\ldots k-1).k$ is 132-avoiding. Suppose that $k\sigma^{k\rtimes 1}(i_1)\sigma^{k\rtimes 1}(i_2)$ is a 132-pattern of $k.\sigma^{k\rtimes 1}(k\ldots n)$ for $k<i_1<i_2$. Knowing the maximality of $t=k-1$, we have to examine two cases.
		\begin{itemize}\setlength\itemsep{-0.3em}
			\item If $\sigma^{-1}(k)> k$, then $\sigma(k)\sigma(i_1-1)\sigma(i_2-1)$ is a 132-pattern for $\sigma$ since $\sigma(k)<k$.
			\item If $\sigma^{-1}(k)\leq k$, then $k\sigma(i_1-1)\sigma(i_2-1)$ is a 132-pattern for $\sigma$.
		\end{itemize}
		This contradict the fact that $\sigma$ is 132-avoiding. Finally, we must have $\sigma^{(k,k)}\in S_{n+1}(132)$.
	\end{proof}
	For any given permutation $\sigma \in S_n$ and an integer $p\geq 1$, we write  $\sigma^{\{(a_1,b_1),\ldots,(a_p,b_p)\}}:=(\ldots(\sigma^{(a_1,b_1)})\ldots)^{(a_p,b_p)} \in S_{n+p}$, where $1\leq a_i,b_i\leq n+i$ for all $i\in [p]$. Furthermore, if $\pi$ is a permutation of length $p$, then we also write  $\sigma^{(a,\pi)}:=\sigma^{\{(a,\pi(1)),(a+1,\pi(2)),\ldots,(a+p-1,\pi(p))\}}$.
	Example:  $3142^{(3,\textcolor{gray}{213})}=3142^{\{(3,\textcolor{gray}{2}), (4,\textcolor{gray}{1}), (5,\textcolor{gray}{3})\}}=41\textcolor{gray}{2}53^{\{ (4,\textcolor{gray}{1}), (5,\textcolor{gray}{3})\}}=52\textcolor{gray}{31}64^{ (5,\textcolor{gray}{3})}=62\textcolor{gray}{413}75$. Let $\alpha$ and $\beta$ be two 132-avoiding permutations. We define the \textit{direct product} of $\alpha$ and $\beta$ as $$\alpha\otimes\beta:=\beta^{(k,\alpha^{+(k-1)})}=\beta^{k \rtimes |\alpha|}(1\ldots k-1).\alpha^{+(k-1)}.\beta^{k \rtimes |\alpha|}(k\ldots |\beta|),$$ where $k=1+|T(\beta)|$.
	Example: $\textcolor{gray}{312}\otimes 543612=543612^{(3, \textcolor{gray}{534})}=87\textcolor{gray}{534}6912$ since $k=3$,  $312^{+2}=534$ and  $543612^{3\rtimes 6}=876912$.
We say that $\sigma$ is \textit{$\oplus$-irreducible}  if it cannot be written as the direct sum of two non-empty permutations. Otherwise, $\sigma$ is called $\oplus$-\textit{decomposable}. Each $\oplus$-decomposable permutation $\sigma$ can be written	uniquely as the direct sum of $\oplus$-irreducible ones, called the \textit{$\oplus$-components} of $\sigma$. Similarly, we also define the notions of   $\otimes$-\textit{irreducible}, $\otimes$-\textit{decomposable} and  \textit{$\otimes$-components} but now over $S(132)$.  We observe that $\crs(\sigma_1\oplus \sigma_2)=\crs(\sigma_1)+\crs(\sigma_2)$ for all permutations $\sigma_1$ and $\sigma_2$. In the following proposition we prove similar result with the new operation $\otimes$ over $S(132)$.
	\begin{proposition}\label{propxxx}
		For all $\sigma_1$ and $\sigma_2\in S(132)$, we have $\crs(\sigma_1\otimes\sigma_2)=\crs(\sigma_1)+\crs(\sigma_2)$.
	\end{proposition}
	\begin{proof}
			Suppose that $\sigma=\sigma_1\otimes\sigma_2$ with $\sigma_1, \sigma_2\in S(132)$. By definition, we have $\sigma(k\ldots k+|\sigma_1|-1)=\sigma_1^{+(k-1)}$ and $\sigma(1\ldots k-1).\sigma(k+|\sigma_1|\ldots |\sigma_1|+|\sigma_2|)=\sigma_2^{k\rtimes |\sigma_1|}$, where $k=1+|T(\sigma_2)|$. Since $\sigma_1^{+(k-1)}$ is a permutation of $\{k,\ldots,|\sigma_1|+k-1\}$, when referring to the arc diagrams of the product $\sigma$, no arrows in $\sigma$ go between an entry in $\sigma_1^{+(k-1)}$  and an entry in $\sigma_2^{k\rtimes |\sigma_1|}$.  Consequently, we have $\crs(\sigma)=\crs(\sigma_1^{+(k-1)})+\crs(\sigma_2^{k\rtimes |\sigma_1|})$. So, the desired result follows from the fact that  $\crs(\pi^{+a})=\crs(\pi)$ and $\crs(\pi^{a\rtimes b})=\crs(\pi)$ for any permutation $\pi$ and integers $a$ and $b$.
	\end{proof}
	
	\begin{proposition}\label{prop:prop2}
		For all  $\sigma_1, \sigma_2\in S(321)$, we have $\Theta(\sigma_1 \oplus \sigma_2)=\Theta(\sigma_2)\otimes\Theta(\sigma_1)$.
	\end{proposition}	
	\begin{proof}
		Let $\sigma=\sigma_1\oplus\sigma_2$ with $\sigma_1 \in S_n(321)$ and   $\sigma_2\in S_m(321)$. Denote first by $(P_1,Q_1)=\RSK(\sigma_1)$. To get $(P,Q)=\RSK(\sigma)$, we have to insert $\sigma_2'=\sigma_2^{+n}$ in $(P_1,Q_1)$. Since all of numbers in $\sigma_2'$ are greater than all of inserted ones in  $(P_1,Q_1)$,  we get $(P,Q)=(P_1.P_2,Q_1.Q_2)$, where $(P_2,Q_2)=RSK(\sigma_2')$ and the dots denote concatenations. So, the Dyck path produced by $(P,Q)$  is $D_1^{(L)}(D_2^{(L)}D_2^{(R)})D_1^{(R)}=\Psi(\sigma)$, where 	
		$D_1^{(L)}$, $D_1^{(R)}$, $D_2^{(L)}$ and $D_2^{(R)}$ are respectively the sub-paths produced by $P_1$, $Q_1$, $P_2$ and $Q_2$.  Moreover, we have  $D_1^{(L)}D_1^{(R)}=\Psi(\sigma_1)$ and $D_2^{(L)}D_2^{(R)}=\Psi(\sigma_2)$. 
		
		Let $k=|D_1^{(R)}|_u+1$. We range in the following table all of assigned numbers to up-steps and down-steps of $\Psi(\sigma)$ for getting $\Theta(\sigma)$.
		
		\begin{table}[h]
			\begin{center}
				\begin{tabular}{|c|c|c|c|}
					\hline
					Sub-path	& $D_1^{(L)}$  &  $ D_2^{(L)}D_2^{(R)}$ & $D_1^{(R)}$\\
					\hline
					For up-steps & $n+m,\ldots,k+m$& $k+m-1,\ldots,k+1,k$& $k-1,\ldots, 2,1$\\
					\hline
					For down-steps &$1,2,\ldots,k-1$ &$k,k+1,\ldots,k+m-1$ &$k+m,\ldots,n+m$\\
					\hline
				\end{tabular}
				\caption{Assigned numbers to up-steps and down-steps of $\Psi(\sigma)$.}
				\label{table:tabx1}
			\end{center}
		\end{table}

		\hspace{-0.68cm}When we compute $\pi=\Phi^{-1}(D_1^{(L)}(D_2^{(L)}D_2^{(R)})D_1^{(R)})=\Theta(\sigma)$ by the described procedure  in section \ref{sec23}, we get two subsequences to be discussed.
		\begin{itemize}
			\setlength\itemsep{-0.3em}
			\item The first one $\pi(k\ldots k+m-1)$ is the sequence produced by $D_2^{(L)}D_2^{(R)}$ and is a permutation of  $\{k,k+1,\ldots k+m-1\}$. In other words, we have $\pi(k\ldots k+m-1)=\pi_2^{+(k-1)}$, where $\pi_2=red[\pi(k\ldots k+m-1)]=\Phi^{-1}(D_2^{(L)}D_2^{(R)})=\Theta(\sigma_2)\in S_{m}(132)$.
			\item The second one  $\pi(1\ldots k-1).\pi(k+m\ldots m+n)$ is the sequence produced by $D_1^{(L)}D_1^{(R)}$ and is a permutation of $[n+m]-\{k,k+1,\ldots k+m-1\}$. When looking at the first and third columns of table 6, we get  $\pi^{-1}(1\ldots k-1)\geq k+m \leq \pi(1\ldots k-1)$. Moreover, if  $\pi_1=red[\pi(1\ldots k-1).\pi(k+m\ldots m+n)]$ then we have $\pi_1^{k\rtimes m}=\pi(1\ldots k-1).\pi(k+m\ldots m+n)$ , $\pi_1=\Phi^{-1}(D_1^{(L)}D_1^{(R)})=\Theta(\sigma_1)\in S_n(132)$. 
		\end{itemize}
		From these two points, we can write $\pi=\pi_1^{k\rtimes m}(1\ldots k-1).\pi_2^{+(k-1)}\pi_1^{k\rtimes m}(k\ldots n)$. Furthermore, we have $k=1+|T(\pi_1)|$ since $|D_1^{(R)}|_u=|T(\pi_1)|$.
		Consequently, using Proposition \ref{prop23}, we have $\pi=\pi_2\otimes \pi_1=\Theta(\sigma_2)\otimes\Theta(\sigma_1)$. This completes the proof of Proposition \ref{prop:prop2}.
	\end{proof}
	It is obvious from Proposition \ref{prop:prop2} that  $\sigma \in S(321)$ is $\oplus$-irreducible if and only if $\Theta(\sigma)\in S(132)$ is $\otimes$-irreducible. 	Moreover, one can easily verify that  $\oplus$  is an associative and stable operation on $S(321)$. As direct consequence of Proposition \ref{prop:prop2}, the direct product $\otimes$ is also stable and associative on $S(132)$.
	
	 Let $\sigma$ be a 321-avoiding permutation. 
	Now, let us adopt the following notations $A_1(\pi, a,b):=|\{b\leq i<a\mid\pi(i)<b\}|$, $A_2(\pi,a,b):=|\{b\leq i<a\mid a<\pi^{-1}(i)\}|$, 
	$A_3(\pi,a,b):=|\{b\leq i<a\mid \pi^{-1}(i)<i<\pi(i)\}|$ \text{ and }$ A_4(\pi,a,b):=|\{b\leq i<a\mid \pi(i)<i<\pi^{-1}(i)\}|$ for any permutation $\pi$ and two integers $a$ and $b$ satisfying $b\leq a\leq |\pi|+1$. We will prove the following lemma which has an important role for later demonstration.
	\begin{lemma} \label{lem:lem31} Let $\pi\in S_{n-1}$ and  $a, b$ two non negative integers satisfying $b\leq a\leq n$. If $\sigma=\pi^{(a,b)}$, we have $	\crs(\sigma)=\crs(\pi)+ \crs(\pi,a,b)$,  where $\crs(\pi,a,b)=A_1(\pi, a,b)+A_2(\pi, a,b)+A_3(\pi, a,b)-A_4(\pi, a,b).$
	\end{lemma}
	\begin{proof}
		Let $\pi \in S_{n-1}$ and $\sigma=\pi^{(a,b)}$ such that $a$ and $b$ satisfy the condition of the lemma. The two subsequences $\sigma(1\ldots b-1)$ and $\pi(1\ldots b-1)$ are in order isomorphic. The same is true for  $\sigma(a+1\ldots n)$ and $\pi(a\ldots n-1)$. Furthermore, if $i\in \{b,\ldots a-1\}$, we have the following properties
		\begin{itemize}
			\setlength\itemsep{-0.3em}
			\item[(a)] it is easy to verify that all crossings of $\pi$, except  lower crossings of the form  $\pi(i)<i<\pi^{-1}(i)$, remain  crossings for $\sigma$, 
			\item[(b)] the new lower arc $(a,b)$ crosses with any arc $(i,\pi(i))$ such that $\pi(i)<b\leq i$
			and with each arc $(i,\pi^{-1}(i))$ such that $i<a<\pi^{-1}(i)$. So $(i,a)$ or $(a,\pi^{-1}(i))$ which is not a crossing of $\pi$ becomes one for $\sigma$, 	
			\item[(c)] if $\pi^{-1}(i)< i<\pi(i)$, then $(\pi^{-1}(i),i)$ becomes a crossing of $\sigma$ since $\pi^{-1}(i)<i< \sigma(\pi^{-1}(i))=i+1<\sigma(i)=\pi(i)+1$.		
		\end{itemize}
		From (a) we get the $\crs(\pi)-A_4(\pi, a,b)$, from (b)  the $A_1(\pi, a,b)+A_2(\pi, a,b)$ and from (c) the $A_3(\pi, a,b)$. Together, that give the desired relation of Lemma \ref{lem:lem31}.
	\end{proof}
	\begin{remark}\label{rem:rem2}
		{\rm Notice that if $b<x<a$, we have $\crs(\pi^{(x,x)},a,b)=\crs(\pi,a-1,b)$.	More generally, if $b<x_1<x_2<\ldots<x_p<a$, we have $\crs(\pi^{\{(x_1,x_1),(x_2,x_2),\ldots, (x_p,x_p)\}},a,b)=\crs(\pi,a-p,b)$}.
	\end{remark}
	
	\begin{lemma} \label{lem:lem33}
		Let $\sigma\in S_n^k(321)$. For all integer $i$ such that $k\leq i <n$, either ($\sigma(i)<k$ and $\sigma^{-1}(i)<i$) or $\sigma^{-1}(i)< i<\sigma(i)$.
	\end{lemma}
	\begin{proof}
		Let  $\sigma\in S_n^k(321)$ and $i$ an integer such that $k\leq i <n$. It is easy to show by contradiction that  if $i$ is an excedance of $\sigma$  then $\sigma^{-1}(i)<i$ and if $i$ is a non-excedance of $\sigma$  then $\sigma(i)<k$ and $\sigma^{-1}(i)<i$.
	\end{proof}
	
	\begin{lemma} \label{lem:lem34}
		Let $\sigma$ be an $\oplus$-irreducible permutation in $S_n^k(321)$ and let $\pi=red[\sigma(1\ldots n-1)]$, ie. $\sigma=\pi^{(n,k)}$. We have $\crs(\Theta(\pi),n-k+j,j)=\crs(\pi,n,k)$, where $j-1$ is the number of matched excedance values  less than $k$. Furthermore, we have the following properties	
		\begin{itemize}
			\setlength\itemsep{-0.3em}
			\item[{\rm (i)}] If $\pi$ is  $\oplus$-irreducible, then $\crs(\pi,n,k)=n-k$. 
			\item[{\rm (ii)}] If $\pi$ is not $\oplus$-irreducible, then there exists $l>1$ such that $\alpha=\pi(1\ldots n-l-1)$ is $\oplus$-irreducible and $\crs(\pi,n,k)=\crs(\alpha,n-l,k)=n-l-k$. 
		\end{itemize} 	
	\end{lemma}
	\begin{proof}
		Let $\sigma$ be an $\oplus$-irreducible permutation in $S_n^k(321)$ and let $\pi=red[\sigma(1\ldots n-1)]$. We have two cases to be considered.
		
		(i) $\pi$ is $\oplus$-irreducible: Using Lemma \ref{lem:lem33}, we obtain $\crs(\pi, n,k)=A_1(\pi,n,k)+A_3(\pi,n,k)=n-k$ since    $A_2(\pi,n,k)=A_4(\pi,n,k)=0$. Since  $\Theta(\pi)$ is also $\otimes$-irreducible and $n-k+j$ is the minimum of non excedance of $\Theta(\sigma)=\Theta(\pi)^{(n-k+j,j)}$, we must have ($\Theta(\pi)^{-1}(i)\geq n-k+j$ and $\Theta(\pi)(i)>i$) or $\Theta(\pi)^{-1}(i)<i< \Theta(\pi)(i)$  for $j\leq i<n-k+j$. That implies $\crs(\Theta(\pi), n-k+j,j)=A_2(\Theta(\pi),n-k+j,j)+A_3(\Theta(\pi),n-k+j,j)=n-k+j-j=n-k$ since $A_1(\Theta(\pi),n-k+j,j)=A_4(\Theta(\pi),n-k+j,j)=0$.
		Hence, we get $\crs(\Theta(\pi),n-k+j,j)=\crs(\pi,n,k)=n-k$. This  ends the proof of the (i) of Lemma \ref{lem:lem34}.
		
		(ii) $\pi$ is $\oplus$-decomposable:  On one hand, there exists an integer  $m>1$ such that  $\pi=\pi_1\oplus \pi_2\oplus\ldots\oplus\pi_{m}$. It is easy to show that $|\pi_i|=1$ for all $i\neq 1$ and $|\pi_1|\geq k$. Consequently, if we denote by $\alpha=\pi_1$ and $n-1-l=|\alpha|$ the length of $\alpha$, then we have $\pi=\alpha \oplus 12\ldots l=\alpha^{\{(n-l,n-l),\ldots,(n-1,n-1)\}}$. So, from Remark \ref{rem:rem2}, we get $\crs(\pi, n,k)=\crs(\alpha, n-l,k)$. Since $\alpha$ is $\oplus$-irreducible, we get $\crs(\pi, n,k)=\crs(\alpha, n-l,k)=n-l-k$. 
		On the other hand, we have $\Theta(\pi)=12\ldots l\otimes\Theta(\alpha)=\Theta(\alpha)^{(i,12\ldots l^{+(i-1)})}$, where $i=1+|T(\Theta(\alpha))|$.  Furthermore, since $n-k+j$ is the minimum of the non-excedances of $\Theta(\sigma)=\Theta(\pi)^{(n-k+j,j)}$, we must have $j<i<i+l-1<n-k+j$. Otherwise, $\Theta(\sigma)$ may not be $\otimes$-irreducible or  not 132-avoiding. Consequently, using again Remark \ref{rem:rem2}, we get $\crs(\Theta(\pi), n-k+j,j)=\crs(\Theta(\alpha), n-l-k+j,j)$. According to the i) of this lemma and  knowing that $\alpha$ is $\oplus$-irreducible, we obtain $\crs(\Theta(\pi), n-k+j,j)=\crs(\Theta(\alpha), n-l-k+j,j)=\crs(\alpha, n-l,k)=n-l-k$. This also ends the proof of the (ii) of Lemma \ref{lem:lem34}.
	\end{proof}
	Now we can prove the $\crs$-preserving of the bijection $\Theta$ that is the main object of this section.
	\begin{theorem} \label{thm:thm31}
		For all $\sigma \in S_n(321)$, we have $\crs(\Theta(\sigma))=\crs(\sigma)$.
	\end{theorem}
	\begin{proof}
		Combining Lemma  \ref{lem:lem31} with Lemma \ref{lem:lem34}, we can proceed by induction of $n$. Theorem \ref{thm:thm31} is obvious for $n =1,2,3$. We assume that Theorem \ref{thm:thm31} holds for $k<n$ and let us consider $\sigma \in S_n(321)$. 
		
		Suppose first that $\sigma$ is $\oplus$-decomposable. We can decompose it as a direct sum $\sigma=\oplus_{i=1}^l\sigma_i$ of $\oplus$-irreducible ones.  From Proposition \ref{prop:prop2}, we get $\Theta(\sigma)=\otimes_{i=1}^l\Theta(\sigma_{l+1-i})$ such that $\Theta(\sigma_i)$ are all $\otimes$-irreducible. Applying the induction hypothesis, we get  $\crs(\Theta(\sigma))=\sum_{k=1}^{l}\crs(\Theta(\sigma_k))=\sum_{k=1}^{l}\crs(\sigma_k)=\crs(\sigma)$. 
		
		Suppose now that $\sigma$ is $\oplus$-irreducible.
		Let $\pi=red[\sigma(1\ldots n-1)] \in S_{n-1}(321)$ and $\Theta(\sigma)=\Theta(\pi)^{(n-k+j,j)}$, where $j-1$ is the number of matched excedance values  less than $\sigma(n)$. When we apply the induction hypothesis with Lemma \ref{lem:lem34}, we get $\crs(\Theta(\sigma))=\crs(\Theta(\pi))+\crs(\Theta(\pi),n-\sigma(n)+j,j)=\crs(\pi)+\crs(\pi,n,\sigma(n))=\crs(\sigma)$. 
		This ends the proof of Theorem \ref{thm:thm31}.	
	\end{proof}
	Combining this result with those of Elizalde and Pak (see Theorem \ref{ElizP}), we get the following one.
	\begin{theorem}\label{thm:thm32}
		The bijection $\Theta$ is $(\fp,\exc,\crs)$-preserving, i.e.~for all $\sigma \in S_n(321)$, we have $(\fp,\exc,\crs)(\Theta(\sigma))=(\fp,\exc,\crs)(\sigma)$. 
	\end{theorem}

	As illustration example to end this section, we draw in Fig. \ref{fig:arcdiag2} the arc diagrams for $\pi=4162735$ and $7652134=\Theta(\pi)$. Observe that we have  $(\fp,\exc,\crs)(\Theta(\pi))=(\fp,\exc,\crs)(\pi)=(0,3,5)$.
	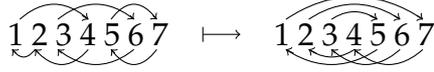
\begin{figure}[h]
		\begin{center}
			\begin{tikzpicture}
			\draw[black] (0,1) node {$1\ 2\ 3\ 4\ 5\ 6\ 7$}; 
			\draw (-0.9,1.2) parabola[parabola height=0.2cm,red] (-0,1.2); \draw[->,black] (-0.08,1.25)--(-0,1.2);
			\draw (-0.4,1.2) parabola[parabola height=0.2cm,red] (0.6,1.2); \draw[->,black] (0.55,1.25)--(0.6,1.2);
			\draw (-0.4,0.8) parabola[parabola height=-0.2cm,red] (0.6,0.8); \draw[->,black] (-0.35,0.75)--(-0.4,0.8);
			\draw (0.3,1.2) parabola[parabola height=0.2cm,red] (0.9,1.2); \draw[->,black] (0.85,1.25)--(0.9,1.2);
			\draw (0.3,0.8) parabola[parabola height=-0.2cm,red] (0.9,0.8); \draw[->,black] (0.35,0.75)--(0.3,0.8);
			
			\draw (-0,0.8) parabola[parabola height=-0.2cm,red] (-0.7,0.8); \draw[->,black] (-0.65,0.75)--(-0.7,0.8);		
			\draw (-1.,0.8) parabola[parabola height=-0.1cm,red] (-0.7,0.8);  \draw[->,black] (-0.95,0.75)--(-1.,0.8);		
			\draw[|->,thin, black] (1.5, 1)--(2,1);		
			\draw[black] (3.5,1) node {$1\ 2\ 3\ 4\ 5\ 6\ 7$}; 
			\draw (2.6,1.2) parabola[parabola height=0.3cm,red] (4.4,1.2); \draw[->,black] (4.32,1.25)--(4.4,1.2);
			\draw (2.9,1.2) parabola[parabola height=0.2cm,red] (4.1,1.2); \draw[->,black] (4.02,1.25)--(4.1,1.2);
			\draw (3.2,1.2) parabola[parabola height=0.1cm,red] (3.8,1.2); \draw[->,black] (3.7,1.25)--(3.8,1.2);
			
			\draw (4.4,0.8) parabola[parabola height=-0.2cm,red] (3.5,0.8);	\draw[->,black] (3.55,0.75)--(3.5,0.8);	
			\draw (4.1,0.8) parabola[parabola height=-0.2cm,red] (3.2,0.8); \draw[->,black] (3.25,0.75)--(3.2,0.8);
			\draw (3.5,0.8) parabola[parabola height=-0.1cm,red] (2.9,0.8); \draw[->,black] (3,0.75)--(2.9,0.8);
			\draw (3.8,0.8) parabola[parabola height=-0.2cm,red] (2.6,0.8); \draw[->,black] (2.7,0.75)--(2.6,0.8);
			
			\end{tikzpicture}
			\caption{Arc diagrams of $\pi=4162735$ and $7652134=\Theta(\pi)$.}
			\label{fig:arcdiag2}
		\end{center}
	\end{figure}
	
	\section{Wilf-equivalence classes modulo $\crs$ and $\nes$}\label{sec4}
	In this section,  we will prove our main result (Theorem \ref{thm:main}) using the bijection $\Theta$ and  some trivial involutions on permutations. These involutions are, the reverse $r$, the complement $c$, and the inverse $i$ such that for any permutation $\sigma$ of $[n]$, 
	\begin{itemize}
		\setlength\itemsep{-0.3em}
		\item[(a)] the \textit{reverse} of $\sigma$ is  $r(\sigma)=\sigma(n)\sigma(n-1)\ldots \sigma(1)$. In other word, $r(\sigma)(j)=\sigma(n+1-j)$ for all $j$. 
		\item[(b)] the \textit{complement} of $\sigma$ is $c(\sigma)=(n+1-\sigma(1))(n+1-\sigma(2))\ldots (n+1-\sigma(n))$. It means that $c(\sigma)(j)=n+1-\sigma(j)$ for all $j$.  
		\item[(c)] the \textit{inverse} of $\sigma$ is $i(\sigma)$, a permutation such that $i(\sigma)(j)=k$ if and only if  $\sigma(k)=j$. We usually denote $\sigma^{-1}=i(\sigma)$.
	\end{itemize}
	To simplify writing, we write $fg:=f\circ g$ for all $f$ and $g$ in $\{r, c, i\}$. Let $\sigma$ be a permutation in $S_n$.
	From the above definitions, the \textit{reverse-complement} of $\sigma$ is $rc(\sigma)$ such that $rc(\sigma)(n+1-i)=n+1-\sigma(i)$ for all $i$ and  the \textit{reverse-complement-inverse} of $\sigma$ is $rci(\sigma)$ such that $rci(\sigma)(n+1-\sigma(i))=n+1-i$ for all $i$. For example, if $\pi=41532$, we have $r(\pi)=23514$, $c(\pi)=25134$, $\pi^{-1}=25413$, $rc(\sigma)=43152$  and $rci(\pi)=35214$. 
	Notice that for any composition $f$ of $r$, $c$ and $i$ and for any subset of permutations $T$, we have the following equivalence
	$$\sigma \in S_n(T) \Longleftrightarrow f(\sigma) \in S_n(f(T)).$$
	This equivalence is well known in the literature, see for example \cite{SiS,West}. As in \cite{Eliz1} and \cite{Dokos}, we need this equivalence to prove most of our results. 
	\begin{lemma}\label{lem:lem41}
		The bijection $rc$ preserves the number of nestings.
	\end{lemma}
	
	\begin{proof}
		Let $\sigma$ be  a permutation in $S_n$.  Using definition of the bijection $rc$, Lemma \ref{lem:lem41} holds from the following facts.
		\begin{itemize}
			\setlength\itemsep{-0.3em}
			\item  $(i,j)$ is an upper nesting of $\sigma \Longleftrightarrow (n+1-j,n+1-i)$ is a lower nesting of $rc(\sigma)$ that does not involve a fixed point,
			\item  $(i,j)$ is a lower nesting of the kind $\sigma(j)<\sigma(i)< i<j$ of $\sigma \Longleftrightarrow (n+1-j,n+1-i)$ is an upper nesting of  $rc(\sigma)$,
			\item The number of upper arcs and the number of lower arcs that embrace a fixed point (as a loop) are always equal. In other words, we have $|\{j<i\mid\sigma(j)>i\}|=|\{j>i\mid \sigma(j)<i\}|$  if $i$ is a fixed point of $\sigma$. 
		\end{itemize}
		Observe that $rc$ exchanges lower and upper arcs of $\sigma$, except loops.
	\end{proof}
	\hspace{-0.58cm}Notice that $rc$ does not preserve the number of crossings. Take as example $\pi=312$ and $rc(\pi)=231$, while $1=\crs(312) \neq \crs(231)=0$. 
	\begin{lemma}\label{lem:lem42}
		The bijection $rci$ preserves the number of crossings and nestings.
	\end{lemma}
	\begin{proof}
	Let $\sigma$ be a permutation of $[n]$.	By definition, we have $rci(\sigma)(n+1-\sigma(i))=n+1-i$ for all $i \in [n]$.
So, the following equivalences immediately hold:
\begin{eqnarray*}
	i<j<\sigma(i)<\sigma(j)\! \! & \! \! \!\Leftrightarrow &\! \! \! n+1-\sigma(j)<n+1-\sigma(i)<n+1-j\! \!<\! \!n+1-i\\
							\! \! \! &\! \! \! \Leftrightarrow & \! \! \! n+1-\sigma(j)\! \!<\! \! n+1-\sigma(i)\! \!<\! \!  rci(\sigma)(n+1-\sigma(j))\! \! <\! \! rci(\sigma)(n+1-\sigma(i)).
\end{eqnarray*}
This means that  $(i,j)$ is an upper-crossing of $\sigma$ if and only if $(n+1-\sigma(j),n+1-\sigma(i))$ is an upper-crossing of $rci(\sigma)$. Similarly, we can easily show the following equivalences
		\begin{itemize}
			\setlength\itemsep{-0.3em}
			\item $(i,j)$ is a crossing of $\sigma \Longleftrightarrow (n+1-\sigma(j),n+1-\sigma(i))$ is a crossing of $rci(\sigma)$;
			\item$(i,j)$ is a nesting of $\sigma \Longleftrightarrow (n+1-\sigma(i),n+1-\sigma(j))$ is a nesting of $rci(\sigma)$.
		\end{itemize}
	Lemma \ref{lem:lem42} follows from these properties.	
	\end{proof}
	Here is an immediate consequence of these two previous lemmas.
	\begin{corollary} \label{thm:thm321} For any subset of patterns $T$, the following statements are true 
		\begin{itemize}
			\setlength\itemsep{-0.3em}
			\item[{\rm i)}] $T$ and $rc(T)$ are $\nes$-Wilf-equivalent,
			\item[{\rm ii)}] $T$ and $rci(T)$ are $(\crs,\nes)$-Wilf-equivalent.
		\end{itemize}
	\end{corollary} 
	\hspace{-0.58cm}Having all the necessary tools, we are now able to prove the following theorem which is equivalent to Theorem \ref{thm:main}.
	\begin{theorem}\label{thm:thm41} For patterns of length 3, we have the following statements.
		\begin{itemize}
			\setlength\itemsep{-0.3em}
			\item[{\rm i)}]The $\nes$-Wilf equivalence classes are: $\{123\}$, $\{321\}$,  $\{213,132\}$ and $\{231, 312\}$,
			\item[{\rm ii)}]The $\crs$-Wilf equivalence classes are:  $\{123\}$, $\{312\}$, $\{231\}$ and $\{132, 213,321\}$,
			\item[{\rm iii)}]The $(\crs,\nes)$-Wilf equivalence classes are:   $\{123\}$, $\{321\}$, $\{312\}$, $\{231\}$ and $\{132, 213\}$.
		\end{itemize}
	\end{theorem}
	\begin{proof}
		Knowing  that $213=rci(132)$ and $312=rc(231)$, from Corollary \ref{thm:thm321}, the patterns 213 and 132 are of the same
		$\crs$, $\nes$ and $(\crs,\nes)$-Wilf equivalence classes and the patterns 231 and 312 are of the same $\nes$-Wilf equivalence classes.  Moreover, following Theorem \ref{thm:thm31} and due to the $\crs$-preserving of the bijection $\Theta$ of Elizalde and Pak, 
		the patterns 321 and 132 are also of the same $\crs$-Wilf-equivalence class. To complete the proof, we can observe Table~\ref{tab:tab41} containing $Nes_n(\tau;x)=\sum_{\sigma \in S_n(\tau)}x^{\nes(\sigma)}$ and $Cr_n(\tau;x)=\sum_{\sigma \in S_n(\tau)}x^{\crs(\sigma)}$ for $n=4$ and $\tau \in S_3$  obtained by simple computation.	
	\end{proof}	
	\begin{table}[h!]
		\begin{center}
			\begin{tabular}{|c|c|c|}
				\hline
				Pattern $\tau$ & $C\!r_4(\tau;x)$ & $N\!es_4(\tau;x)$\\
				\hline
				123& $7+6x+x^2$ &  $4+8x+2x^2$\\
				132,213&$8+4x+2x^2$ & $7+5x+2x^2$\\
				321& $8+4x+2x^2$ & 14\\
				231& $8+5x+x^2$ & $8+5x+x^2$\\
				312& $13+x$ & $8+5x+x^2$\\
				\hline
			\end{tabular}
			\caption{$C\!r_4(\tau;x)$ and  $N\!es_4(\tau;x)$ for $\tau \in S_3$.}
			\label{tab:tab41}
		\end{center}
	\end{table}
	
	So, when we combine our result with those of Elizalde and Dokos et al.~, we get the following one.
	\begin{theorem} \label{thm:thm42} For pattern in $S_3$,  the non singleton Wilf-equivalence classes are
		\begin{itemize}
			\setlength\itemsep{-0.3em}
			\item[{\rm i)}] $[132]_{\fp,\exc,\inv,\crs,\nes}=\{132,213\}$,
			\item[{\rm ii)}]  $[231]_{\fp,\inv,\nes}=\{231,312\}$,
			\item[{\rm iii)}] $[132]_{\fp,\exc,\crs}= \{132,213,321\}$.
		\end{itemize}
	\end{theorem}
	\begin{proof}
		For i) and ii), it is easy to show that the bijection $rci$ preserves all statistics in $\{\fp, \exc,\inv,\crs, \nes\}$ and the bijection $rc$ preserves all statistics in $\{\fp, \inv, \nes\}$. Especially for iii), we use the $(\fp,\exc,\crs)$-preserving of the bijection $\Theta$.
	\end{proof}
	\section{Connection to the $q,p$-Catalan number of Randrianarivony} \label{sec5}
	In this section, we present an unexpected result on the joint distribution of the statistics $\exc$ and $\crs$ (see Theorem \ref{thm:thm52}). 
	The connection with the q,p-Catalan number of Randrianarivony is due to the following lemma  about the characterization of nonnesting permutations in terms of avoiding permutations. 
	\begin{lemma} \label{lem:lem51}For all integer n, we have $N\!N_n =S_n(321)$.
	\end{lemma}
	\begin{proof}
		It is clear that a given permutation $\sigma$ is nonnesting if and only if it is bi-increasing. Following  Reifegerste \cite{ARef}, bi-increasing permutations and 321-avoiding permutations are the same.
	\end{proof}
	
	Randrianarivony \cite{ARandr} defined a $q,p-$Catalan numbers $C_n(q,p)$ through the relation
	\begin{equation} \label{eq:cat}
	C_n(q,p)=C_{n-1}(q,p)+q\sum_{k=0}^{n-2}p^kC_k(q,p)C_{n-1-k}(q,p)
	\end{equation}
	with $C_0(q,p)=C_1(q,p)=1$ and he purposed some combinatorial interpretations of $C_n(q,p)$ in terms of noncrossing and nonnesting permutations. According to Lemma \ref{lem:lem51}, one of its proved results can be stated as follow.
	\begin{theorem} {\rm \cite{ARandr}}\label{thm:randr} For all integer $n\geq 0$, we have $\displaystyle \sum_{\sigma \in S_n(321)}q^{\exc(\sigma)}p^{\crs(\sigma)}=C_n(q,p)$.
	\end{theorem}
	\begin{proof}
		See \cite{ARandr}.
	\end{proof}
	\hspace{-0.58cm}Notice that the iii) of Theorem \ref{thm:thm42} is equivalent to the following identities.
	\begin{equation}\label{eq:gener}
	\sum_{\sigma \in S_n(213)}x^{\fp(\sigma)}q^{\exc(\sigma)}p^{\crs(\sigma)}=\sum_{\sigma \in S_n(132)}x^{\fp(\sigma)}q^{\exc(\sigma)}p^{\crs(\sigma)}=\sum_{\sigma \in S_n(321)}x^{\fp(\sigma)}q^{\exc(\sigma)}p^{\crs(\sigma)}.
	\end{equation}
	The case p=1 of (\ref{eq:gener}) was treated in \cite{ElizD, ElizP, Eliz1, Eliz2}. With Theorem \ref{thm:randr}, the case $x=1$ leads to the following one.
	\begin{theorem}\label{thm:thm52}
		For all non negative integer $n$ and for all $\tau \in \{213, 132, 321\}$, we have
		\begin{equation*}
		\sum_{\sigma \in S_n(\tau)}q^{\exc(\sigma)}p^{\crs(\sigma)}=C_n(q,p).
		\end{equation*}
	\end{theorem}	
	\begin{corollary}
		For any $\tau \in \{213, 132, 321\}$, the continued fraction expansion of $\displaystyle \sum_{\sigma \in S(\tau)}q^{\exc(\sigma)}p^{\crs(\sigma)}z^{|\sigma|}$ is 
		\begin{equation}
		\frac{1}{1-\displaystyle\frac{z}{				
				1-\displaystyle\frac{qz}{
					1-\displaystyle\frac{pz}{								
						1-\displaystyle\frac{qpz}{
							1-\displaystyle\frac{p^2z}{
								1-\displaystyle\frac{qp^2z}{				
									\ddots}
							}
					}}		
		}}}. \label{eq:fc}
		\end{equation}
	\end{corollary}
	\begin{proof}
		Let us denote by $\displaystyle C(q,p,z):=\sum_{\sigma \in S(\tau)}q^{\exc(\sigma)}p^{\crs(\sigma)}z^{|\sigma|}$ for  $\tau \in \{213, 132, 321\}$. According to Theorem \ref{thm:thm52}, we have $\displaystyle C(q,p,z)=\sum_{n\geq 0}C_n(q,p)z^n$. So, using recurrence  (\ref{eq:cat}), we get the following identity 	which leads to (\ref{eq:fc})
		\begin{eqnarray*}
			C(q,p,z)&=&\frac{1}{1-\displaystyle\frac{z}{1-qzC(q,p,pz)}}.
		\end{eqnarray*}
	\end{proof}	
	
	Let us end this section with an interesting remark on the recursion formula for the polynomial distribution of the statistic $\inv$ over the set  $S_n(321)$.
	If we denote by $I_n(q)=\sum_{\sigma \in S_n(321)}q^{\inv(\sigma)}$, we have the following recurrence formula which was first conjectured in~\cite{Dokos}
	\begin{equation}\label{eq:inv}
	I_n(q)=I_{n-1}(q) +\sum_{k=0}^{n-2}q^{k+1}I_k(q)I_{n-1-k}(q).
	\end{equation}
	Using other objects like 2-Motzkin paths and polyominoes, this recursion was later proved by Cheng et al.~in \cite{CEKS}.  Based on the scanning-elements algorithm, Mansour and Shattuck \cite{MansShatt} also provided another proof. 
	We observe that (\ref{eq:inv}) can be obtained from (\ref{eq:cat}) by setting $p=q$, i.e.~$I_n(q)=C_n(q,q)$. Indeed,  it was proved in \cite{MedVienot, ARandr} that $\inv(\sigma)=2\nes(\sigma)+\crs(\sigma)+\exc(\sigma)$ for all permutation $\sigma$. So, if $\sigma \in S_n(321)$,  then we have $\inv(\sigma)=\crs(\sigma)+\exc(\sigma)$ and we get $I_n(q)=C_n(q,q)$. The continued fraction  expansion of the generating function of $I_n(q)$ presented in \cite{MansShatt} (Theorem 1) is also obtained from (\ref{eq:fc}) by setting $p=q$.
	
	\section{Concluding remarks} \label{sec6}
	We conclude this paper with two remarks. The first one is about the decomposition of Dyck path involving centered multitunnel and the second one is about the direct bijection $\Gamma : S_n(321) \rightarrow S_n(132)$ defined by A. Robertson \cite{ARob2}.
	
	According to the original definition of Elizalde and Deutsch, let us remember what a multitunnel is. A \textit{multitunnel} of  a Dyck paths $D$ is a concatenation of tunnels in which each tunnel starts at the point where the previous one ends. Centered multitunnels are those  whose midpoints stay on the vertical line $x=n$. For example, the Dyck path $D=ududuuuddudduudd$ in Fig.~\ref{fig:dyckstat} has three centered multitunnels. As mentioned in \cite{ElizD}, each centered multitunnel is in obvious one-to-one correspondence with decomposition of the Dyck word $D = ABC$ where $B$ and $AC$ are Dyck paths,  $B$ is the section that runs along the entire multitunnel, $A$ and $C$ have the same length. So, an operation $\circledcirc$ defined by $D_1\circledcirc D_2=D_1^{(L)}D_2D_1^{(R)}$ is well defined and stable on Dyck paths. Inspiring from the proof of Proposition \ref{prop:prop2}, we remark that there is a correspondence between the three operations  $\oplus$, $\circledcirc$  and $\otimes$. 
	\begin{remark}
		{\rm For all $\sigma_1, \sigma_2 \in S(321)$, we have $\sigma_2\oplus\sigma_2\longmapsto \Psi(\sigma_1)\circledcirc\Psi(\sigma_2)$ and for any Dyck paths $D_1$ and $D_2$, we also have $D_1\circledcirc D_2\longmapsto \Phi^{-1}(D_2)\otimes\Phi^{-1}(D_1)$}. 
	\end{remark}

	In terms of statistic, the number of $\oplus$-components of $\sigma$, the number of centered multitunnels of $\Psi(\sigma)$ and the number of $\otimes$-components of $\Theta(\sigma)$ are the same for any $\sigma \in S(321)$. We illustrates this correspondence by an example in Figure \ref{fig:321block}. 
	\begin{figure}[h]
		\begin{center}
			\begin{tikzpicture}
			\draw[black] (0,0.2) node {$\textcolor{gray}{2413}\oplus \textcolor{blue}{1}\oplus 312$};
			\draw[|->,thin, black] (1.5, 0.2)--(2.2, 0.2);
			\draw[gray] (2.5, 0)--(2.75, 0.25)--(3, 0)--(3.25, 0.25)--(3.5, 0)--(3.75, 0.25)--(4, 0.5)--(4.25, 0.25)--(4.5, 0);
			\draw[black] (4.75,0.2) node {$\circledcirc$};
			\draw[blue](5, 0)--(5.25, 0.25)--(5.5, 0);
			\draw[black] (5.75,0.2) node {$\circledcirc$};	
			\draw[black](6, 0)--(6.25, 0.25)--(6.5, 0.5)--(6.75, 0.25)--(7, 0)--(7.25, 0.25)--(7.5, 0);
			\draw[|->,thin, black] (7.75, 0.2)--(8.5,0.2);
			\draw[black] (10.2,0.2) node {$312$  $\otimes$   \textcolor{blue}{1}  $\otimes$  \textcolor{gray}{3}\textcolor{gray}{4}\textcolor{gray}{2}\textcolor{gray}{1}};
			\end{tikzpicture}
			\caption{Example of correspondence between $\oplus$, $\circledcirc$ and $\otimes$.} \label{fig:321block}
		\end{center}
	\end{figure} 
	
	Let us recall the direct bijection $\Gamma : S_n(321) \rightarrow S_n(132)$ defined by A. Robertson. Let  $\sigma(i)\sigma(j)\sigma(k)$ and $\sigma(x)\sigma(y)\sigma(z)$ be two occurrences of pattern 132 in a permutation $\sigma$. Say that $\sigma(i)\sigma(j)\sigma(k)$ is smaller than $\sigma(x)\sigma(y)\sigma(z)$ if $(i, j, k) \prec (x, y, z)$, where $\prec$ denote the lexicographic ordering of triples of positive integers. Let $\mathcal{M}$ be an operation that creates the permutation $\mathcal{M}\sigma$ from $\sigma$ by converting the smallest occurrence of 132-pattern into 321-pattern. It is clear that $\mathcal{M}\sigma=\sigma$ if $\sigma$ is 132-avoiding. We denote by $\mathcal{M}^{j}\sigma=\mathcal{M} \mathcal{M}^{j-1}\sigma$, $j\geq 1$. It is known that, for every $\sigma \in S_n$, $\mathcal{M}^{j}\sigma$ must be 132-avoiding for some finite $j$. So, we can define the bijection $\Gamma$ as follow
	\begin{equation*}
	\Gamma(\sigma)=
	\begin{cases}
	\sigma &\text{ if  } \sigma \in S_n(132);\\
	\mathcal{M}^{r}\sigma &\text{ if } \sigma \notin S_n(132).
	\end{cases}
	\end{equation*}
	where $r$ is the smallest positive integer such that $\mathcal{M}^{r}\sigma$ is 132-avoiding.\\
	For example, if $\pi=4162735 \in S_7(321)$, then we have   	
	$\mathcal{M} \pi=6152734$, 
	$\mathcal{M}^{2} \pi=6521734$, 
	$\mathcal{M}^{3} \pi=6571324$ and 
	$\mathcal{M}^{4} \pi=6573214=\Gamma(\pi)$. We remark that the following result holds and seems to be interesting.
	\begin{theorem}
		The bijection $\Gamma$ is also {\rm (\fp,\exc,\crs)}-preserving.
	\end{theorem} 
	\hspace{-0.58cm}Indeed, Bloom  and Saracino \cite{Bloom,Bloom2} proved  that the bijection $\Gamma$ is $(\fp,\exc)$-preserving. Recently,  Saracino \cite{Sarino} showed how the bijections $\Theta$ and $\Gamma$ are related  each other by a simple relation. He proved the following theorem.
	\begin{theorem}{\rm \cite{Sarino}} We have
		$\Gamma(\sigma)=\Theta\circ rci(\sigma)$ for any $\sigma \in S_n(321)$, where $rci$ is the reverse-complement-inverse.
	\end{theorem}
	\hspace{-0.58cm}Since the bijections $\Theta$ and  $rci$ are both \crs-preserving, so do the bijection $\Gamma$.
	Below is a graphical illustration example to close this paper. For $\pi=4162735$ we have $\crs(\Gamma(\pi))=\crs(\pi)=5$.
	\begin{figure}[h]
		\begin{center}
			\begin{tikzpicture}
			\draw[black] (0,1) node {$1\ 2\ 3\ 4\ 5\ 6\ 7$}; 
			\draw (-0.9,1.2) parabola[parabola height=0.2cm,red] (-0,1.2); \draw[->,black] (-0.08,1.25)--(-0,1.2);
			\draw (-0.4,1.2) parabola[parabola height=0.2cm,red] (0.6,1.2); \draw[->,black] (0.55,1.25)--(0.6,1.2);
			\draw (-0.4,0.8) parabola[parabola height=-0.2cm,red] (0.6,0.8); \draw[->,black] (-0.35,0.75)--(-0.4,0.8);
			\draw (0.3,1.2) parabola[parabola height=0.2cm,red] (0.9,1.2); \draw[->,black] (0.85,1.25)--(0.9,1.2);
			\draw (0.3,0.8) parabola[parabola height=-0.2cm,red] (0.9,0.8); \draw[->,black] (0.35,0.75)--(0.3,0.8);
			
			\draw (-0,0.8) parabola[parabola height=-0.2cm,red] (-0.7,0.8); \draw[->,black] (-0.65,0.75)--(-0.7,0.8);		
			\draw (-1.,0.8) parabola[parabola height=-0.1cm,red] (-0.7,0.8);  \draw[->,black] (-0.95,0.75)--(-1.,0.8);		
			\draw[|->,thin, black] (1.5, 1)--(2,1);		
			\draw[black] (3.5,1) node {$1\ 2\ 3\ 4\ 5\ 6\ 7$}; 
			
			\draw (2.6,1.2) parabola[parabola height=0.3cm,red] (4.1,1.2); 
			\draw[->,black] (4.05,1.25)--(4.1,1.2);
			\draw (2.9,1.2) parabola[parabola height=0.2cm,red] (3.8,1.2); \draw[->,black] (3.75,1.25)--(3.8,1.2);
			\draw (3.2,1.2) parabola[parabola height=0.2cm,red] (4.4,1.2); \draw[->,black] (4.3,1.25)--(4.4,1.2);		
			\draw (4.4,0.8) parabola[parabola height=-0.2cm,red] (3.5,0.8);	\draw[->,black] (3.55,0.75)--(3.5,0.8);	
			\draw (4.1,0.8) parabola[parabola height=-0.3cm,red] (2.6,0.8); \draw[->,black] (2.65,0.75)--(2.6,0.8);
			\draw (3.5,0.8) parabola[parabola height=-0.1cm,red] (3.2,0.8); \draw[->,black] (3.25,0.75)--(3.2,0.8);
			\draw (3.8,0.8) parabola[parabola height=-0.2cm,red] (2.9,0.8); \draw[->,black] (3,0.73)--(2.9,0.8);		
			\end{tikzpicture}
			\caption{Arc diagrams of $\pi=4162735$ and $6573214=\Gamma(\pi)$.}
			\label{fig:arcdiag3}
		\end{center}
	\end{figure}
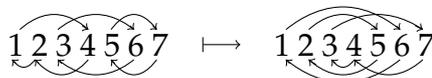
	
	\hspace{-0.58cm}\textbf{Acknowledgments}\\
	We would like to thank Arthur Randrianarivony  for suggesting this problem  and anonymous referees for carefully reading  and helpful comments that much improve the content of this paper.

\end{document}